\documentclass[preprint,review,11pt]{elsarticle}
\usepackage{subfig}

\usepackage{amsmath,amssymb}
\usepackage{graphicx}
\usepackage{url} 
\usepackage{booktabs}
\usepackage{times,bm}
\usepackage{dcolumn}
\usepackage{algorithm,algorithmicx,algpseudocode,epstopdf}
\usepackage[left=2.5cm,right=2.5cm,top=2.5cm,bottom=2.5cm]{geometry}
\usepackage{amsthm}

\usepackage[T1]{fontenc}
\usepackage{textcomp}
\usepackage{amsmath}
\usepackage{url}
\usepackage{booktabs}

\usepackage[widespace]{fourier}

\newcommand{\mcP}{\mathcal{P}}
\newcommand{\mcR}{\mathcal{R}}

\newcommand{\Nc}{\mathcal{N}_C}
\newcommand{\Nr}{\mathcal{N}_r}
\newcommand{\No}{\mathcal{N}_\omega}
\newcommand{\mcL}{\mathcal{L}}
\newcommand{\mcH}{\mathcal{H}}
\newcommand{\mcW}{\mathcal{W}}
\newcommand{\mcQ}{\bm{\mathcal{Q}}}
\newcommand{\mcC}{\mathcal{C}}
\newcommand{\mcV}{\mathcal{V}}
\newcommand{\mcT}{\mathcal{T}}
\newcommand{\mcI}{\mathcal{I}}

\newcommand{\bfg}{{\boldsymbol g}}
\newcommand{\bff}{{\boldsymbol f}}
\newcommand{\bfp}{{\boldsymbol p}}
\newcommand{\bfP}{{\boldsymbol P}}
\newcommand{\bfq}{{\boldsymbol q}}
\newcommand{\bfR}{{\boldsymbol R}}
\newcommand{\bfe}{{\boldsymbol e}}
\newcommand{\bfpsi}{{\boldsymbol \psi}}
\newcommand{\bfphi}{{\boldsymbol \phi}}
\newcommand{\bftphi}{\widetilde{\boldsymbol \phi}}
\newcommand{\bfPhi}{{\boldsymbol \Phi}}

\newcommand{\bfzero}{{\boldsymbol 0}}
\newcommand{\bfpi}{{\boldsymbol \pi}}

\newcommand{\diam}{\text{diam}}

\newcommand{\realmr}{\mathbb{R}^{\Nr}}
\newcommand{\poly}{\mathbb{P}}
\newcommand{\mean}{{\operatorname{mean}}}
\newcommand{\dotphi}{\dot{\phi}}
\newcommand{\tildedotphi}{\dot{\widetilde{\phi}}}
\newcommand{\dotbfphi}{\dot{\bfphi}}
\newcommand{\dotbftphi}{\dot{\widetilde{\bfphi}}}
\newcommand{\dotbfe}{{\dot{\bfe}}}

\newcommand{\ATOL}{\text{ATOL}}
\newcommand{\LATOL}{\text{LATOL}}
\newcommand{\xMAX}{\text{xMAX}}
\newcommand{\tMAX}{\text{tMAX}}
\newcommand{\sMAX}{\text{sMAX}}
\newcommand{\new}{\text{new}}

\renewcommand{\AE}{\widetilde{E}}
\providecommand{\abs}[1]{\lvert#1\rvert}
\providecommand{\norm}[1]{\lVert#1\rVert}

\providecommand{\scalar}[1]{\langle#1\rangle}

\numberwithin{equation}{section}
\numberwithin{figure}{section}
\numberwithin{table}{section}

\newtheorem{theorem}{Theorem}[section]

\newtheorem{definition}{Definition}[section]

\journal{CMAME}

\begin{document}


\begin{frontmatter}

\title{Adaptive Finite Element Solution of Multiscale PDE-ODE Systems}
\date{\today}

\author[aj]{ A.~Johansson}
\ead{august@simula.no}
\author[jc]{J.~H.~Chaudhry}
\ead{jehanzebh@hotmail.com}
\author[vc]{V.~Carey}
\ead{varis@ices.utexas.edu}
\author[de]{D.~Estep\corref{cor1}}
\ead{estep@stat.colostate.edu}
\author[vg]{V.~Ginting}
\ead{vginting@uwyo.edu}
\author[ml]{M.~Larson}
\ead{mats.larson@math.umu.se}
\author[st]{S.~Tavener}
\ead{tavener@math.colostate.edu}

\cortext[cor1]{Corresponding author}

\address[aj]{Center for Biomedical Computing, Simula Research Laboratory, P.O. Box
134, 1325 Lysaker, Norway}
\address[jc]{Department of Scientific Computing, Florida State University, Tallahassee, FL 32306}
\address[vc]{Institute for Computational Engineering and Sciences, University of Texas at Austin,
Austin, TX 78712}
\address[de]{Department of Statistics, Colorado State University, Fort Collins, CO 80523}
\address[vg]{Department of Mathematics, University of Wyoming, Laramie, WY 82071}
\address[ml]{Dept.of Mathematics, Umea University, S-90187 Umea, Sweden}
\address[st]{Department of Mathematics, Colorado State University, Fort Collins, CO 80523}

\begin{abstract}
  We consider adaptive finite element methods for solving a multiscale system consisting of a macroscale model
  comprising a system of reaction-diffusion partial differential
  equations coupled to a microscale model comprising a system of
  nonlinear ordinary differential equations. A motivating example is modeling the electrical activity of the
  heart taking into account the chemistry inside cells in the
  heart. Such multiscale models pose extremely computationally
  challenging problems due to the multiple scales in time and space that
  are involved.

  We describe a mathematically consistent approach to couple the
  microscale and macroscale models based on introducing an intermediate
  ``coupling scale''. Since the ordinary differential
  equations are defined on a much finer spatial
  scale than the finite element discretization
  for the partial differential equation, we introduce a Monte Carlo approach to sampling the fine scale ordinary differential equations.    We derive goal-oriented a posteriori error
  estimates for quantities of interest computed from the solution of the multiscale model using adjoint problems and computable
  residuals. We distinguish the errors in time and space for the partial
  differential equation and the ordinary differential equations
  separately and include errors due to the transfer of the solutions
  between the equations. The estimate also includes terms reflecting the sampling of the microscale model. Based on the accurate error estimates, we
  devise an adaptive solution method  using a ``blockwise''
  approach. The method and estimates are
  illustrated using a realistic problem.
\end{abstract}

\begin{keyword}
  {\em a posteriori} error analysis, adaptive error control,
  adaptive mesh refinement, adjoint problem, coupled physics, duality, generalized Green's function, goal oriented error estimates, multiscale model, residual, variational analysis
\end{keyword}

\end{frontmatter}

\section{Introduction}

Our interest in problems consisting of a macroscale time dependent partial
differential equation (PDE) coupled to a miscroscale system of  ordinary
differential equations (ODEs) originates in the modeling of the
electrical activity in the heart. The standard macroscale model of
 electrical phenomena in cardiac tissue is the bidomain
model proposed by Tung \cite{tungphd}, which consists of parabolic
and elliptic PDEs modeling the macroscopic potential
distribution. These PDEs are derived assuming a representation of
the tissue as two anisotropic media, one intracellular, which is
strongly anisotropic, and one extracellular, which is weakly
anisotropic. If the two media are assumed to have proportional
conductivity tensors, it is possible to reduce the model to the
monodomain model, consisting of a single reaction-diffusion PDE with a
load depending on the solution to the ODEs.

On the cellular level, the electrical activity may be modeled by a set
of ODEs that depends on a potential determined by the solution of the
PDE. Many cellular models are available, both phenomenological models
that try to mimic measurements, and physiological models which are
based on measurements as well as physiological theory. The latter may
be very complex, involving up to hundred variables.  For a review on mathematical models describing the electrical activity in the heart, see Sundnes et.~al.~\cite{heartbook} and the references
therein. A space-time adaptive method can be found in Colli Franzone
and coworkers \cite{collifranzone}. A survey of heart modeling can be
found in Noble \cite{noble_modeling}.

Specialized numerical methods are required for high fidelity simulation of the heart, since the heart may consist of up to $10^{10}$ cells
\cite{heartbook}, each modeled by a set of ODEs, and it is impossible
to solve the PDE on the same spatial scale as the ODEs. Moreover, determining the actual physical geometry and location of cells is itself a difficult problem. Thus, including the cellular scale phenomena in the macroscale discretization requires some form of ``upscaling'' or ``recovery'' of the information provided by the microscale modeling. This is in fact necessary if the PDE
model derived using homogenization as is the case for
the bidomain equations.

Unfortunately, this subtle mathematical issue is often ignored in  computational electrocardiography, where it is common to simply evaluate the ODEs in the cells located at the quadrature points of a finite element method, for example. However, this is not mathematically consistent, since caused the model to change with the PDE discretization. One consequence, for example, is that it is impossible to perform a mesh convergence study, which is the crudest form of uncertainty quantification.

As an alternative, we create an intermediate ``coupling'' or ``mesoscale'' representation of the cellular scale physics that is used to exchange information between the macroscale and microscale. To deal with the very large number of cells, we create the mesoscale representation by sampling cells at the microscale at random and taking averages over the mesoscale cells.

Another potential issue in such coupled systems is significant differences in the temporal scales in the different components. For example, in a system where the ODEs model chemical reactions and the PDE models global behavior such as transport,  it is likely that the dynamics
of the chemical systems are much faster than that of
the total system. Coupled PDE-ODE systems where the ODEs
describe chemical reactions and the PDE describe transport occur in
applications such as the study of pollution in groundwater, surface
water and the atmosphere, control theory and semiconductor simulation. To deal with this, we allow the PDE and ODE systems to employ significantly different time steps.

Numerical solutions of such multiscale systems are invariably affected by error arising from numerous discretization effects and present significant discretization challenges in terms of obtaining a desired accuracy \cite{estepms}. It is therefore critically important to accurately estimate the numerical error in computed quantities of interest and devise efficient discretization parameter selection algorithms. In this paper, we derive goal-oriented a posteriori error estimates that distinguish the relative contributions of various discretization effects, and thus provide the capability of adjusting various discretization parameters to efficiently obtain a desired accuracy. The error analysis is based on a posteriori error estimates that employ computable residuals and adjoint equations, see \cite{actaintro,svartadoden,memoir,BangerthRannacher,giles_suli} for general information. For applications to multiscale systems, see \cite{estepms,victoropsplit,CET09,Logg1,gintingode,CEGT12a,vet13}. We base the adaptive strategy on the block adaptive approach described in \cite{blockadap}.

The content of this paper is organized as follows: In Section 2, we
formulate the multiscale model. In Section 3, we describe the discretization methods. The
a posteriori error analysis is presented in Section 4. In Section 5, we describe some implementation details and the adaptive algorithm. A numerical example is presented in Section 6. The paper ends with a conclusion in Section 7.

\section{Model description}

The nominal model problem consists of a reaction-dif\-fusion PDE that describes macroscale behavior over a
domain $\Omega$  coupled to systems of ODEs that model processes taking place inside small
``cells'' $C$ that comprise the heart domain.  The coupling of
the macro- and micro-scale processes taking place on vastly different
scales in space and time raise serious challenges for analyzing the
behavior and computing solutions of the model. We first describe the
original coupled system, then we describe a new system that includes a coupling
mechanism that provides an avenue to address these challenges.

\subsection{The original model}

The region $\Omega= \cup_{i=1}^{\Nc} C_i$ is comprised of $\Nc$ cells
$C_i$ indexed as $\{1, \ldots, \Nc\}$. We model the {\em microscale}
behavior using a collection of ODEs: Find $\bfp_i \in
[\mcC^1(0,T)]^{\Nr}$ solving
\begin{equation}\label{origchemistry}
  \begin{cases}
    \dot{\bfp_i} = \bfg_i(u;\bfp_i), & t \in (0,T], \\
      \bfp_i(0) = \bfp_i^0, &
  \end{cases}, \quad i = 1, \ldots, \Nc,
\end{equation}
where $\bfp_i$ is a vector of length $\Nr$ and $u$ is the solution of
a PDE modeling the macroscopic behavior described below. In the
context of (\ref{origchemistry}), $u$ has the role of a parameter. We
have allowed the model for the microscale behavior $\bfg_i$ to vary
with each cell. For simplicity of notation, we have assumed the same
number of equations in each cell model, however this is not necessary.

In order to introduce the microscale solutions of the reactants into
the macroscale model, we define the piecewise constant function
$\bfp(x)$ for $x\in \Omega$,
\begin{equation*}
  \bfp(x,t) = \bfp_i(t), \quad (x,t) \in  C_i \times (0,T], \quad i = 1, \ldots, \Nc.
\end{equation*}
The {\em macroscale} model problem reads: Find $u(x,t) \in
\mcC^1((0,T); \mcC^2(\Omega))$ such that
\begin{equation}  \label{origPDE}
  \begin{cases}
    \dot{u} - \nabla \cdot \epsilon \nabla u = f(u; \bfp), & (x,t) \in \Omega \times (0,T],\\
      n \cdot \epsilon \nabla u = 0,  & (x,t) \in \partial \Omega \times (0,T], \\
        u(x,0) = u^0(x), & x \in \Omega,
  \end{cases}
\end{equation}
where $\epsilon=\epsilon(x) \geq \epsilon_0>0$ is a continuous function and
$f$ and $u^0$ are sufficiently smooth functions. In this equation, $\bfp$
now plays the role of a parameter, but one that varies in space on the
microscale.

\subsection{Multiscale coupling}
\label{sect:multiscale_coupling}

The coupled system (\ref{origchemistry})-(\ref{origPDE}) immediately
raises several issues:
\begin{itemize}
\item The solutions of the microscale ODEs (\ref{origchemistry})
  vary in space on the scale of the cells. This happens both because
  of varying cell model and because the ODE model
  (\ref{origchemistry}) depend on experimentally-determined
  parameters that vary stochastically. This microscale variation
  introduces extremely rapid variation in the coefficient $f$ on the
  scale of the macroscale PDE (\ref{origPDE}) along with
  discontinuities across cell boundaries. We would therefore have to
  use a spatial discretization for (\ref{origPDE}) that is finer than
  the cells while being consistent with the cell boundaries in order
  to achieve full order accuracy in numerical solutions.
\item The ODE system (\ref{origchemistry}) has an extremely large
  dimension $\Nc\times\Nr$, with the consequence that the solution of
  many systems of nonlinear ODEs are required to advance the PDE
  solution if we solve the microscale model (\ref{origchemistry}) in
  every cell. This raises another significant computational burden.
\item At the same time, we expect to see a macroscale pattern in
  variations in cell type and model, which implies it is inefficient
  to integrate the microscale ODEs in every cell.
\end{itemize}

To deal with these issues, we introduce a {\em coupling scale}
decomposition of $\Omega$. We assume that $\Omega = \cup_{j=1}^{\No}
\omega_j$ is decomposed into a set of non-overlapping regions
$\omega_i$. Each $\omega_j$ is comprised of a collection of cells
$\omega_j = \cup_{i \in \mcI_j} C_i$, where $\mcI_j$ is a subset of
the indices $\{1, \ldots, \Nc\}$. We assume the collection
$\{\mcI_j\}$ is non-intersecting while their union equals $\{1,
\ldots, \Nc\}$. To smooth out the cell-scale variation in the reaction
model (\ref{origchemistry}), we average the reaction solutions over
$\omega_i$. We introduce a ``recovery'' operator $\mcR :
\mathbb{R}^{\Nc\times\Nr} \to [\mcL^2(\Omega)]^{\Nr}$ defined as
\begin{equation}\label{recovery}
  \mcR \bfp (x) = \frac{1}{|\omega_j|} \int_{\omega_j} \bfp(y) \, dy
  = \frac{1}{|\omega_j|} \left( \sum_{i \in \mcI_j} \bfp_i |C_i| \right),
  \quad x \in \omega_j, \quad j = 1, \ldots, \No ,
\end{equation}
where $|\omega_j|$ and $|C_i|$ denote the volume of the indicated
region and cell respectively. Note that $\sum_{i \in \mcI_j} |C_i| =
|\omega_j|$. The function $\mcR \bfp$ is piecewise constant, but now
varies on the coupling scale rather than the cell scale.

One reasonable criteria to choosing the intermediate scale cells is to assume that the same reaction model $\bfg_i$
is used for each cell $C_i$ in each coupling region $\omega_j$.  We
now replace the original {\em macroscale} model problem
(\ref{origPDE}) by
\begin{equation}  \label{PDE}
  \begin{cases}
    \dot{u} - \nabla \cdot \epsilon \nabla u = f(u; \mcR \bfp), & (x,t) \in \Omega \times (0,T],\\
      n \cdot \epsilon \nabla u = 0,  & (x,t) \in \partial \Omega \times (0,T], \\
        u(x,0) = u^0(x), & x \in \Omega,
  \end{cases}
\end{equation}

Next we note that in the original formulation, $\bfp_i$ depends
implicitly on the spatial variable $x$ (which has the role of a
parameter) {\em inside} each cell. To avoid this, we introduce a
projection of $u$ into a space of functions that are constant on each
cell. We let $\mcP: \mcL^2(\Omega) \to \mathbb{R}^{\Nc}$ be a suitably
chosen projection into functions that are piecewise constant on the
cells and we replace (\ref{origchemistry}) by
\begin{equation}\label{chemistry}
  \begin{cases}
    \dot{\bfp_i} = \bfg_i(\mcP u; \bfp_i ), & t \in (0,T], \\
      \bfp_i(0) = \bfp_i^0, &
  \end{cases} \quad i = 1, \ldots, \Nc,
\end{equation}
When the exact spatial location of each cell is unavailable, as often
is the case, we use a projection $\mcP$ into the space of functions
that are constant on the coupling scale domains $\omega_j$, which also
produces a function that is constant on each cell.

\section{Multirate finite element methods}
\label{sect:fem}

In order to derive a variational a posteriori error estimate, we write
the time discretization as a finite element method for a piecewise
polynomial while using a common finite element method for spatial
discretization. Combined with suitable quadrature formulas, the
resulting approximations match standard finite difference schemes.

\subsection{Variational formulation}
To this end, we let $(\cdot,\cdot)_X$ denote the inner product on
$\mcL^2(X)$ on a space $X$ with corresponding norm $\norm{\cdot}_X$
and let $a(v,w)$ denote the bilinear form $a(v,w)_X= (\epsilon \nabla
v,\nabla w)_X$. The subscript $X$ is dropped when
$X=\Omega$. Furthermore, we let $\scalar{\cdot,\cdot}_{C_j}$ denote
the inner product on $\realmr$ on cell $C_j$ and let
$\scalar{\cdot,\cdot} = \sum_{i=1}^{\Nc}
\scalar{\cdot,\cdot}_{C_i}$. The corresponding norm is denoted by
$\norm{\cdot}$, which is the same notation as for the $\mcL^2$-norm,
but it is obvious from the context which norm is
intended. Furthermore, we write the right hand side
functions $f$ and $\bfg$ as functions of two variables, replacing
'$;$' with '$,$'.

We first write (\ref{chemistry})-(\ref{PDE}) in
variational form: The solutions $\bfp_i \in [\mcH^1(0,T)]^{\Nr}$ of
(\ref{chemistry}) satisfy,
\begin{equation}
  \label{chemistryvar}
  \int_{0}^T \scalar{\dot{\bfp}_i,\bfq} \ dt
  = \int_{0}^T \scalar{\bfg_i(\mcP u,\bfp_i),\bfq} \ dt, \quad \forall
  \bfq \in [\mcL^2(0,T)]^{\Nr}, \; \quad i = 1, \ldots, \Nc,
\end{equation}
while the solution $u \in \mcL^2((0,T);\mcH^1(\Omega))$ satisfies,
\begin{equation}
  \label{PDEvar}
  \int_{0}^T (\dot{u},v) + a(u,v) \ dt  =
  \int_{0}^T (f(u,{\mcR} \bfp),v) \ dt, \quad \forall v \in \mcL^2((0,T);\mcH^1(\Omega)).
\end{equation}

\subsection{A multirate finite element method}\label{section:multirate}
We solve the coupled system (\ref{chemistry})-(\ref{PDE}) using a
discretization that allows different time steps to be used for the
ODEs and the PDE. The discretization yields a nonlinear coupled
system of discrete equations for the approximate solution that must be solved iteratively in practice. It is common to fix the number of iterations used for such coupled systems, which can significantly affect the properties of the resulting numerical solution. In the extreme case
with no iteration, this represents a so-called explicit-implicit
scheme.

For the temporal discretization for the PDE, the time interval $[0,T]$
is partitioned into $N$ subintervals $0 = t_0< t_1 < \cdots < t_N=T$,
and we denote each subinterval by $I_n=(t_{n-1},t_n]$ with length by
  $\Delta t_n = t_n - t_{n-1}$.  For the temporal mesh for the ODEs,
  each interval $I_n$ may be divided into $M_n$ subintervals by
  $t_{n-1}=s_{0,n} <s_{1,n} < \cdots < s_{M_n,n} = t_n$, where
  $J_{m,n} = (s_{m-1,n},s_{m,n}]$ and is of length $\Delta s_{m,n} =
    s_{m,n} - s_{m-1,n}$. This is illustrated in Fig.
    \ref{fig:timestepping}. We also allow for different time steps in
    different cells, but the notation becomes very cumbersome so we do
    not indicate this. Moreover, we note that it is possible to let
    the individual $\Nr$ ODE components have different time stepping
    as in \cite{Logg1}, but this is not considered here.

\begin{figure}
  \begin{center}
    \includegraphics[scale=0.5]{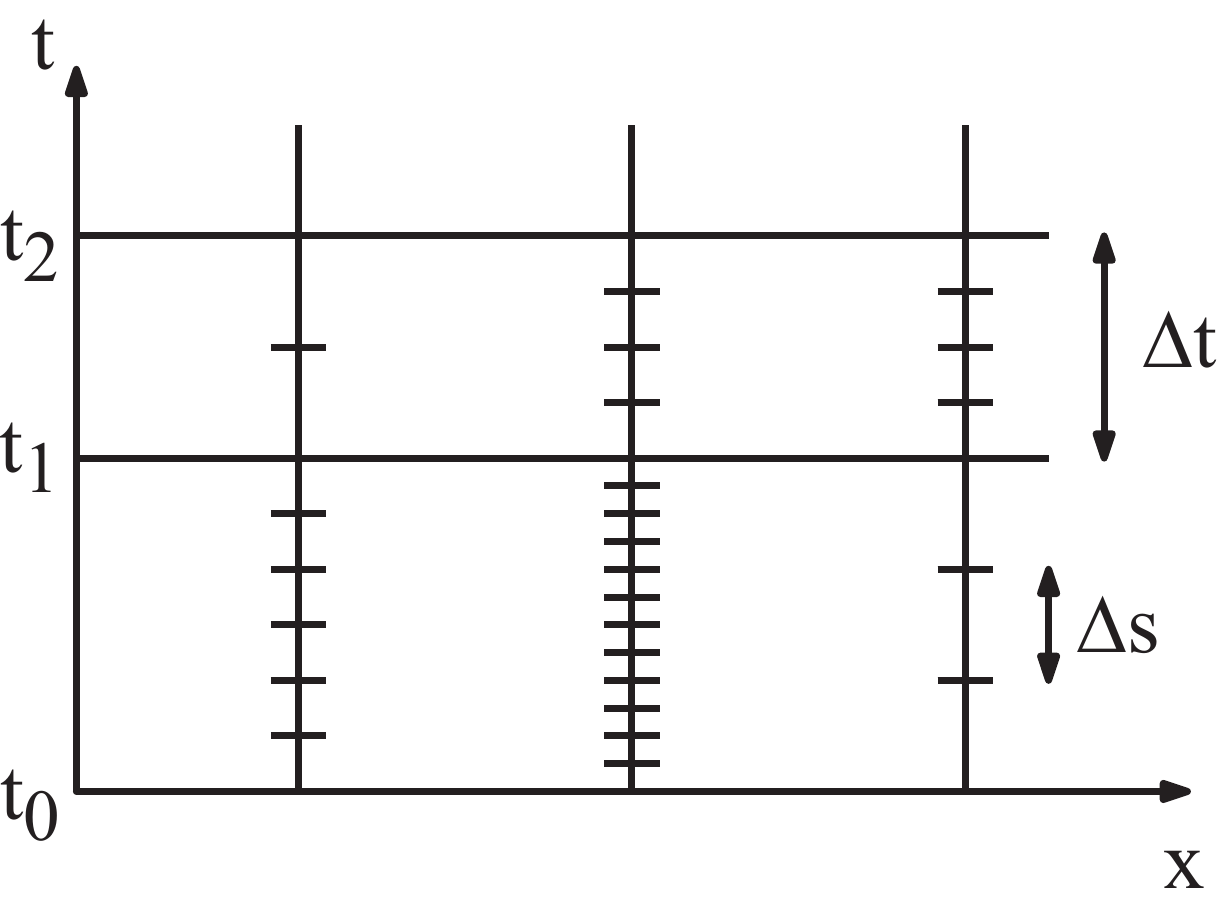}
    \caption{The PDE time steps $\Delta t$ (long horizontal lines) is
      allowed to depend on $t$ but not on $x$. The ODE time steps
      $\Delta s$ (short horizontal lines) may depend on both $t$ and
      $x$.}
    \label{fig:timestepping}
  \end{center}
\end{figure}

The space of polynomials of order $r$ is denoted $\poly^r$, and we
discretize all the $\Nr$ ODE components in the space of polynomials of
degree $r_p$, $ \mcQ_{m,n}^{r_p} = \{ \bfq(t) : \bfq|_{I_{m,n}} \in
[\poly^{r_p}(I_{m,n})]^{\Nr} \}$. We denote the space of functions
$\cup_{m=1}^{M_n} \mcQ_{m,n}^{r_p}$ by $\mcQ_n^{r_p}$. To simplify
notation, we use the same order of polynomial in all cells, though
this is not necessary.

For the PDE, we define a triangulation $\mcT_n^h$ of $\Omega$ on each
interval $I_n$ that is inconsistent with the coupling scale
decomposition $\Omega = \cup_{j=1}^{\No} \omega_j$: We let $\mcT_n^h$
be hexahedral elements and let the coupling scale partition be the
Voronoi tessellation defined by $\No$ uniformly distributed random
points in $\Omega$, i.e.\ each $\omega_j$ is a Voronoi cell
\cite{Aurenhammer}. An illustration of this tessellation can be seen in
Fig. \ref{fig:vorocube}.

Let $\mcV_n^h\subset \mcH^1(\Omega)$ be the space $ \mcV_n^h = \{ v(x)
\in \mcC(\Omega) : v|_K \in \poly^r(K), \ K\in \mcT_n^h \}$. To
indicate the sizes of the elements in $\mcV_n^h$, we introduce the
mesh function $h_K=\diam(K)$ for $K\in \mcT_n^h$ and $h = \max h_K$. The approximation is
in the space-time function space $ \mcW^{r_u}_n = \{ w(x,t) : w(x,t) =
\sum_{i=0}^{r_u} t^i v_i(x),\ v_i \in \mcV_n^h,\ (x,t) \in \Omega
\times I_n\}$.

\begin{figure}
  \begin{center}
    \includegraphics[scale=0.75]{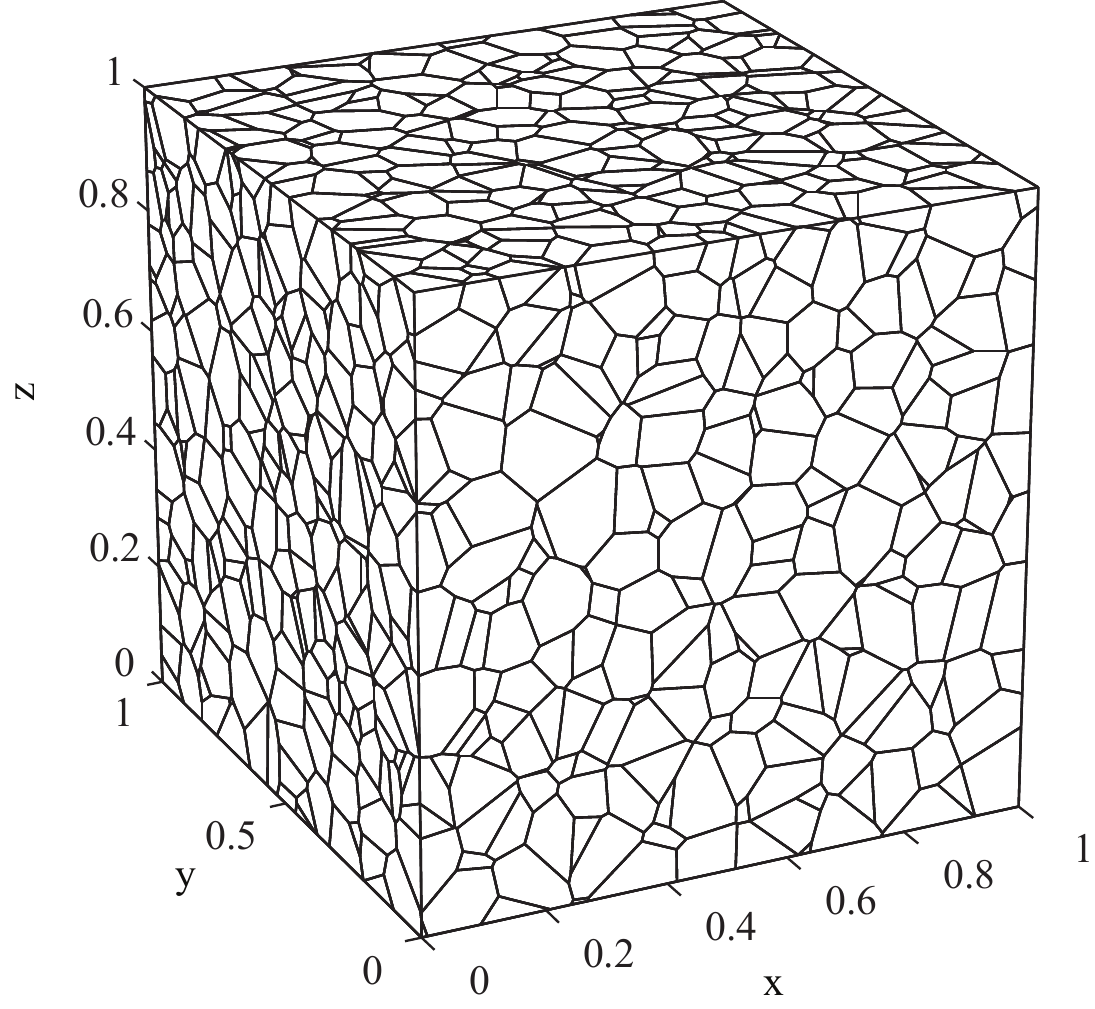}
    \caption{Illustration of a Voronoi tessellation with $\No=4096$.}
    \label{fig:vorocube}
  \end{center}
\end{figure}

The functions in $\mcW^{r_u}_n$ and $\mcQ^{r_p}_n$ are discontinuous
at time nodes, and we denote the jump across a time nodes $t_i$ by
$[v]_i = v_i^+-v_i^-$, where $v_i^\pm = \lim_{t\rightarrow t_i^\pm}
v(t)$. Finally in order to evaluate Galerkin orthogonality, we use the
following projection operators into the discrete spaces:
\begin{align} 
  \Pi_n&: \mcL^2(\Omega) \rightarrow \mcV_n^h, \\
  \pi_n^u &: \mcL^2(I_n) \rightarrow \poly^{r_u}(I_n), \\
  \bfpi_{m,n}^p &: [\mcL^2(I_{m,n})]^{\Nr} \rightarrow [\poly^{r_p}(I_{m,n})]^{\Nr}.
\end{align}
Note that $\pi_n^u \Pi_n = \Pi_n \pi_n^u : \mcL^2(\Omega \times I_n)
\rightarrow \mcW^{r_u}_n$.

The multirate finite element method now reads: For all cells in $C$
and for each time interval $n=1,\ldots, N$, find $\bfP \in
\mcQ_{m,n}^{r_p}$ for $m=1,\ldots,M_n$ such that
\begin{align}
  \label{Pdgimplicit}
  \int_{I_{m,n}} \scalar{\dot{\bfP},\bfq} \ dt + \scalar{[\bfP]_{m-1,n},\bfq^+}
  = \int_{I_{m,n}} \scalar{\bfg(\mcP U,\bfP),\bfq} \ dt, \quad \forall
  \bfq \in \mcQ_{m,n}^{r_p},
\end{align}
with $\bfP^{0-} = \bfp^0$. Note that there are $\Nc$ ODE systems in
(\ref{Pdgimplicit}). The PDE discretization is: Find $U \in
\mcW^{r_u}_n$ such that
\begin{align}
  \label{Udgimplicit}
  \int_{I_n} (\dot{U},v) + a(U,v) \ dt + ([U]_{n-1},v^+) &=
  \int_{I_n} (f(U,{\mcR} \bfP),v) \ dt, \quad \forall v \in \mcW^{r_u}_n,
\end{align}
where $U_0^{-}=\Pi_0 u^0$.

Note that if we use $r_p=r_u=0$ with a left-hand rectangle rule for
time integrals and employ a standard ``lumped-mass'' quadrature rule
in space (trapezoidal rule on elements), then we obtain a standard
difference approximation consisting of the implicit Euler in time and
5 point stencil difference scheme in space \cite{memoir}.

\subsection{Practical discretization considerations}\label{section:iterative}

There are additional considerations for discretization that are used
in practice.

\subsubsection*{Evaluation of the recovery operator}

The definition of the recovery operator $\mcR$ in (\ref{recovery})
requires the solution of the cell reaction equations (\ref{chemistry})
on all $\Nc$ cells. We expect $\Nc$ to be very large, e.g. on the
order of millions to billions. At the same time, we expect that the
physical properties of cells vary over the macroscale rather than
microscale, so it is inefficient to solve (\ref{chemistry}) on every
cell.  We approximate $\mcR$ by an average $\widetilde{\mcR}$ computed
using a Monte Carlo sampling approach with a (relatively small) sample
of the problems (\ref{chemistry}). As a consequence, we replace the
PDE equation (\ref{Udgimplicit}) by: Find $U \in \mcW^{r_u}_n$ such
that
\begin{align}
  \label{Udgimplicit_R}
  \int_{I_n} ( \dot{U},v) + a(U,v) \ dt + ([U]_{n-1},v^+) =
  \int_{I_n} (f(U,\widetilde{\mcR} \bfP),v) \ dt, \quad \forall v \in \mcW^{r_u}_n,
\end{align}
where $U_0^{-}=\Pi_0 u^0$.

There are two ways to compute the samples. First, we can sample in the
spatial variable $x$, so for $x \in \omega_j$,
\begin{equation}\label{MCrecovery1}
  \widetilde{\mcR} \bfp (x) =  \frac{1}{K_j} \sum_{k=1}^{K_j} \bfp( x_{j,k}),
\end{equation}
where $\{x_{j,k}\}_{k=1}^{K_j}$ is a set of $K_j$ points chosen at
random in $\omega_j$, for each $j = 1, \ldots, \No$. In this approach,
we solve the reaction model (\ref{chemistry}) only in cells located at
$\{x_{j,k}\}_{k=1}^{K_j}$. This means that the average is affected by
by physical size of the cells, since bigger cells contribute more
samples with some probability.  Alternatively, we can sample by
choosing cell indices at random, so for $x \in \omega_j$,
\begin{equation}\label{MCrecovery2}
  \widetilde{\mcR} \bfp (x) =  \frac{1}{K_j} \sum_{k=1}^{K_j} \bfp_{i_k},
\end{equation}
where $\{i_k\}$ is a randomly selected subset of $\mcI_j$. In the
second approach, the cell average is not affected by the size of the
cell.

Note that this Monte Carlo computation has the property that the computation is actually exact if sufficiently many samples are used, since there
are only $\Nc$ different values. However for very large $\Nc$ and
reasonably small numbers of samples, the convergence of the Monte
Carlo approximations appears to behave according to the standard asymptotic results, so
that the accuracy is roughly proportional to the variance of the
integrand divided by the square root of the number of samples. Since
we want to use relatively few samples, this indicates that the
coupling regions $\omega_j$ should be chosen so that the variance of
the integrand defining $\mcR$ has small variance.

\subsubsection*{Iterative solution of the nonlinear discrete equations}

In general, (\ref{Pdgimplicit})-(\ref{Udgimplicit_R}) presents a
nonlinear coupled discrete system that is solved iteratively in
practice. We describe a simple fixed point iteration.

We use the superscript $\ell$ to denote the iteration number and
assume that we carry out $L_n$ total iterations on each time step. In
practice, we may vary the number of iterations in each cell and for
the PDE, but we suppress that notation.  The fixed point - finite
element method can be formulated on time interval $I_n$ as: Given
$\bfP^0=\bfP(t_{n-1})$ and $U^0=U(x,t_{n-1})$, for each cell find
$\bfP^\ell \in \mcQ_{m,n}^{r_p}$ for $m=1,\dots,M_n$ on $I_n$ such
that
\begin{multline}
  \label{discPim}
  \int_{I_{m,n}} \scalar{\dot{\bfP^\ell},\bfq} \ dt + \scalar{(\bfP^{\ell+}-\bfP^-)|_{m-1,n},\bfq^+_{m-1,n}} = \\
  \int_{I_{m,n}} \scalar{\bfg(\mcP U^{\ell-1},\bfP^\ell),\bfq} \ dt, \quad \forall \bfq \in
  \mcQ_{m,n}^{r_p}.
\end{multline}
Then, given $\bfP^\ell$ on $I_n$, find $U^\ell \in \mcW^{r_u}_n$ such
that
\begin{multline}
  \label{discUim}
  \int_{I_n} (\dot{U}^\ell,v) + a(U^\ell,v) \ dt + ((U^{\ell+}-U^-)|_{n-1},v^+_{n-1}) \\
  =
  \int_{I_n} (f(U^\ell,\widetilde{\mcR} \bfP^\ell),v) \ dt, \quad \forall v \in \mcW^{r_u}_n,
\end{multline}
given $U^0$. We iterate \eqref{discPim} and \eqref{discUim} for
$\ell=1,\ldots,L_n$, set $\bfP^-_{n}= \bfP^{L_n}(t_n)$ and $U^-_n = U^
{L_n}(t_n)$ to form the data for the next interval.

In this iterative formulation, we are ``lagging'' the values of the
PDE model $U$ in the cell reaction model, but treat the remaining
nonlinear problems for $\bfP^\ell$ and $U^\ell$ implicitly. This
entails using an additional nonlinear solver for each model
equation. However, we do not indicate this in the notation as these
problems can typically be solved very accurately, e.g. using a
standard Newton method.

\subsubsection*{An implicit-explicit method}

In practice, the iterative formulation (\ref{discPim})-(\ref{discUim})
is often solved for only one iteration $L_n=1$, yielding an
``implicit-explicit'' method. This is: For $n=1,\ldots, N$, find $\bfP
\in \mcQ_{m,n}^{r_p}$ for $m=1,\ldots,M_n$ such that
\begin{align}
  \label{Pdg}
  \int_{I_{m,n}} \scalar{\dot{\bfP},\bfq} \ dt + \scalar{[\bfP]_{m-1,n},\bfq^+}
  = \int_{I_{m,n}} \scalar{\bfg(\mcP U_{n-1},\bfP),\bfq} \ dt, \quad \forall
  \bfq \in \mcQ_{m,n}^{r_p},
\end{align}
with $\bfP^{0-} = \bfp^0$. Then find $U \in \mcW^{r_u}_n$ such that
\begin{align}
  \label{Udg}
  \int_{I_n} (\dot{U},v) + a(U,v) \ dt + ([U]_{n-1},v^+) =
  \int_{I_n} (f(U,\widetilde{\mcR} \bfP),v) \ dt, \quad \forall v \in \mcW^{r_u}_n,
\end{align}
where $U_0^{-}=\Pi_0 u^0$. The scheme is said to be explicit-implicit
due to the use of explicit use of $U_{n-1}$ when solving for $\bfP$ on
$I_n$, while $U_n$ is solved implicitly.

\section{A posteriori error analysis}

In this section, we derive adjoint-based error representation formulas
for the methods above. Using these formulas, we derive indicators of local contributions to the global error that can serve as the basis for an adaptive method. Ignoring the effect of iteration in the solution of the discrete equations, there are three discretization parameters that affec the numerical accuracy, namely
the spatial mesh size, the time steps for the PDE, and the time steps
for the ODEs. In addition, there is a choice of the projection and
recovery operators. We present error indicators that distinguish the relative contributions of these choices to the global error. For the iterative method,
there is the additional parameter of number of iterations per time
step.

We begin by noting that since the right hand side of \eqref{Pdg}
involves the solution $U_{n-1}$ from the previous time step and since
the right hand side of \eqref{Udg} involves the approximation
$\widetilde{\mcR}$, there is an error arising from operator decomposition \cite{estepms} due the
differences $U$ and $U_{n-1}$ as well as a modeling error due to the
differences between $\mcR$ and $\widetilde{\mcR}$.

\subsection{Preliminaries}

The adjoint to the function spaces presented is denoted by the
superscript $*$. The same notation is  used for the adjoint
operators, as noted in the following identities
\begin{align}
  (v,\mcR \bfq) &= \scalar{ \mcR^* v,\bfq }, \\
  \scalar{\mcP v,\bfq} &= ( v,\mcP^* \bfq ),
\end{align}
which hold for any $\bfq\in \mcQ_{m,n}^{r_p}$ and $v\in \mcW_n^{r_u}$.

We define the residuals, the
linearizations of the functions $f$ and $\bfg$ and state the adjoint
problems. The derivation of the theorems follow in the next Section.

\begin{definition}\label{def:residuals}
  Let the residuals corresponding to \eqref{Pdg} and \eqref{Udg} be
  denoted by $\bfR_p(\bfP) \in \mcQ_{m,n}^{r_p *}$ and $R_{u,K}(U,\bfP)
  \in \mcW^{r_u *}_n$. These are defined by
  \begin{align}
    \label{resp}
    \int_{I_{m,n}} \scalar{\bfR_p(\bfP),\bfq} \ dt
    &= \int_{I_{m,n}} \scalar{\bfg(\mcP U_{n-1},\bfP) - \dot{\bfP},\bfq} \ dt -
    \scalar{[\bfP]_{m-1,n},\bfq_{m-1,n}^+}, \\
    \label{resu}
    \int_{I_n} (R_{u,K}(U,\bfP),v)_K \ dt
    &= \int_{I_n} (f(U,\widetilde{\mcR}\bfP) - \dot{U} + \nabla
    \cdot \epsilon \nabla U, v)_K \nonumber \\ &\qquad -
    \frac{1}{2}([n\cdot \epsilon \nabla U],v)_{\partial K\setminus
      \partial \Omega} \ dt
    - ([U]_{n-1},v_{n-1}^+)_K.
  \end{align}
\end{definition}

\begin{definition}\label{def:linearization}
  The linearizations of $f$ and $\bfg$ are defined as
  \begin{align}
    \overline{f_u} &= \int_0^1 \frac{\partial f}{\partial u} (us+U(1-s),\mcR \bfp s+\mcR \bfP (1-s)) \ ds,\\
    \overline{\bff_p} &= \left[ \int_0^1 \frac{\partial f}{\partial \mcR \bfp_j} (us+U(1-s),\mcR \bfp s+\mcR \bfP (1-s)) \ ds \right]_{j=1}^{\Nr},\\
    \overline{\bfg_u}  &= \left[ \int_0^1 \frac{\partial \bfg_i}{\partial \mcP u} (\mcP us+\mcP U(1-s),\bfp s+\bfP (1-s)) \ ds \right]_{i=1}^{\Nr}, \\
    \overline{\bfg_p} &= \left[ \int_0^1 \frac{\partial \bfg_i}{\partial \bfp_j} (\mcP us+\mcP U(1-s),\bfp s+\bfP (1-s)) \ ds \right]_{i,j=1}^{\Nr}.
  \end{align}
\end{definition}

\subsection{Adjoint problems}
The definition of an appropriate adjoint problem for a multiscale model is a problematic issue \cite{estepms}. First of all, there are a number of possible adjoint problems that can be associated with any nonlinear differential equation. In addition, it may be reasonable to account for the use of projections between scales and representations and the effects of a finite number of iterations in the definition. Another issue is the computational difficulty and cost required in the numerical solution of the adjoint.

We consider two different adjoint problems that present a tradeoff between accuracy of the resulting estimates on one hand and the cost and practicality of
implementation on the other. The first
approach is closely related to the ideal adjoint of the coupled
system~\eqref{chemistryvar}-\eqref{PDEvar} treated by an implicit discretization, a so-called ``implicit adjoint''. This choice leads to a robustly accurate error estimate, but the cost is that the linear adjoint problem is nearly as difficult to solve numerically as the original multiscale model.  As an alternative, we define a second adjoint problem that uses the same decomposition and finite iterations as used in the forward model.  This approach is much more computationally tractable, however the resulting error estimate includes terms that cannot be estimated. It is possible to show that these terms are relative small compared to the terms that can be estimated in the limit of refined discretization however.

\subsection{Fully implicit adjoint problem and error representation formula}
The ideal fully implicit adjoint reads: For $n=N,\ldots,1$, find $\bftphi_p(t) \in
[\mcL^2(I_{m,n})]^{\Nr}$ on $m=1,\ldots,M_n$ and
$\widetilde{\phi}_u(x,t) \in \mcL^2(I_n;\mcH^1(\Omega))$ such that
\begin{align}
  \label{adjointP_FI}
  \int_{I_{m,n}} \scalar{\bfq,-\dotbftphi_p}  - \scalar{\bfq,\overline{\bfg_p}\hspace{0pt}^* \bftphi_p + \mcR^* \overline{\bff_p}\hspace{0pt}^*  \widetilde{\phi}_{u}} \ dt &= \int_{I_{m,n}} \scalar{\bfq,\bfpsi_p} \ dt
\end{align}
and
\begin{align}
  \label{adjointU_FI}
  \int_{I_n} (v,-\tildedotphi_u) + a(v,\widetilde{\phi}_u)  - (v,\overline{f_u} \hspace{0pt}^* \widetilde{\phi}_u + \mcP^* \overline{g_u}\hspace{0pt}^* \bftphi_p) \ dt &= \int_{I_n} (v,\psi_u) \ dt,
\end{align}
where $\psi_u\in \mcL^2(\Omega)$ and $\bfpsi_p \in
[\mcL^2(\Omega)]^{\Nr}$ are given data that determine the quantity of interest to be computed. The initial conditions
for the adjoint problem are $\widetilde{\phi}_u(x,T) = 0$ and
$\bftphi_p(T) = \mathbf{0}$.  Next we derive the error representation
formula corresponding to this fully implicit adjoint problem.

\begin{theorem}[Error representation formula]\label{lem:errrep_fi}
  Let the quantity of interest be the linear functional $m(u,\bfp)$
  be defined by the functions $\psi_u\in \mcL^2(\Omega)$ and
  $\bfpsi_p \in [\mcL^2(\Omega)]^{\Nr}$ such that $m(u,\bfp) = \int_0^T
  (u,\psi_u) + \scalar{ \bfp,\bfpsi_p } \ dt$. The error representation
  formula for the error $E(U,\bfP) = \abs{ m(u,\bfp) - m(U,\bfP)}$ reads
  \begin{align}
    \label{errorformula2_fi}
    E(U,\bfP)
    &=  \Big| \sum_{n=1}^N \int_{I_n} (e_u,\psi_u) + \scalar{\bfe_p,\bfpsi_p} \ dt \Big| \\
    &= \Big| I + \sum_{n=1}^N ( II_n + III_n + IV_n + V_n + V\!I_n) \Big|,
  \end{align}
  where
  \begin{align}
    \label{errorI_fi}
    I &=  (u^0-\Pi_0 u^0,\widetilde{\phi}_{u,0}), \\
    \label{errorII_fi}
    II_n &= \sum_{K\in \mcT_n^h} \int_{I_n} (R_{u,K}(U,\bfP), \widetilde{\phi}_u - \Pi_n \pi_n^u \widetilde{\phi}_u)_K \ dt, \\
    \label{errorIII_fi}
    III_n &= \sum_{m=1}^{M_n} \int_{I_{m,n}} \scalar{\bfR_p(\bfP), \bftphi_p - \bfpi_{m,n}^p \bftphi_p} \ dt, \\
    \label{errorIV_fi}
    IV_n &= \int_{I_n} (f(U,\mcR \bfP) - f(U,\widetilde{\mcR}\bfP), \widetilde{\phi}_u) \ dt, \\
    \label{errorV_fi}
    V_n &=  \int_{I_n} \scalar{ \bfg(\mcP U,\bfP) - \bfg(\mcP U_{n-1},\bfP), \bftphi_p } \ dt.
  \end{align}
\end{theorem}

The first term is the contribution from error in the initial data for the PDE. The second
and the third terms quantify the contributions of the discretization of the
PDE and the ODEs respectively. The fourth term quantifies the contribution of the recovery
operator and the fifth term is the contribution of the explicit splitting
scheme.
\begin{proof}
  Introducing the errors $e_u = u-U$ and $\bfe_p = \bfp - \bfP$, we
  use the chain rule identities to obtain,
  \begin{align}
    \scalar{\overline{\bfg_u} \mcP e_u + \overline{\bfg_p} \bfe_p, \bfq} &=
    \scalar{ \bfg(\mcP u,\bfp)-\bfg(\mcP U,\bfP), \bfq}, \label{crid_1}\\
    (\overline{f_u}e_u + \overline{\bff_p} \mcR \bfe_p,v) &=
    (f(u,\mcR \bfp)-f(U,\mcR \bfP),v) \label{crid_2}.
  \end{align}
  Note that \eqref{crid_1} holds on each $I_{m,n}$ whereas
  \eqref{crid_2} holds on each $I_n$. Furthermore, by the continuity of
  $u$,
  \begin{equation}
    \begin{aligned}
      e_{u,(m-1,n)}^+ &= u_{m-1,n}^+ - U_{m-1,n}^+ = (u_{m-1,n}^- -  U_{m-1,n}^-) - (U_{m-1,n}^+ - U_{m-1,n}^-) \\&= e_{u,(m-1,n)}^- - [U]_{m-1,n}.
    \end{aligned}
  \end{equation}
  Similarly,
  \begin{equation}
    \label{eq:err_p_jump}
    \begin{aligned}
      \bfe_{p,(m-1,n)}^+ &= \bfe_{p,(m-1,n)}^- - [\bfP]_{m-1,n}.
    \end{aligned}
  \end{equation}
  Now substituting $\bfe_p$ for $\bfq$ in \eqref{adjointP_FI} and
  applying integration by parts we arrive at,
  \begin{equation}
    \begin{aligned}
      \int_{I_{m,n}} \scalar{\bfe_p,\bfpsi_p} \ dt &= \int_{I_{m,n}} \scalar{\bfe_p ,-\dotbftphi_p}  - \scalar{\bfe_p ,\overline{\bfg_p}\hspace{0pt}^* \bftphi_p + \mcR^* \overline{\bff_p}\hspace{0pt}^*  \widetilde{\phi}_{u}} \ dt \\
      &= -\scalar{\bfe_p^- ,\bftphi_p}_{m-1,n} +  \scalar{\bfe_p^+ ,\bftphi_p}_{m-1,n}  \\
      & \phantom{a} \quad + \int_{I_{m,n}} \scalar{\dotbfe_p ,\bftphi_p}  - \scalar{\overline{\bfg_p} \bfe_p , \bftphi_p} \ dt -  \int_{I_{m,n}}\scalar{ \overline{\bff_p} \mcR \bfe_p , \widetilde{\phi}_{u}} \ dt
    \end{aligned}
  \end{equation}
  Now we apply \eqref{eq:err_p_jump},
  \begin{equation}
    \label{eq:err_ep}
    \begin{aligned}
      \int_{I_{m,n}} \scalar{\bfe_p,\bfpsi_p} \ dt &= -\scalar{\bfe_p^- ,\bftphi_p}_{m,n} + \scalar{\bfe_p^- ,\bftphi_p}_{m-1,n}  - \scalar{[\bfP]_{m-1,n},\bftphi_{p,(m-1,n)}}\\
      & \phantom{a} \quad + \int_{I_{m,n}} \scalar{\dotbfe_p ,\bftphi_p}  - \scalar{\overline{\bfg_p} \bfe_p , \bftphi_p} - \scalar{ \overline{\bff_p} \mcR \bfe_p , \widetilde{\phi}_{u}} \ dt.
    \end{aligned}
  \end{equation}
  A similar computation for $e_u$ leads to,
  \begin{equation}
    \label{eq:err_eu}
    \begin{aligned}
      \int_{I_{n}} (e_u,\psi_u) \ dt &=  -(e_u^-,\widetilde{\phi}_u)_{n} + (e_u^-,\widetilde{\phi}_u)_{n-1} - ([U]_{n-1},\widetilde{\phi}_{u,(n-1)}) \\&
      + \int_{I_{n}} (\dot{e}_u,\widetilde{\phi}_u)
      + a(e_u,\widetilde{\phi}_u)  - (\overline{f_u} e_u,\widetilde{\phi}_u )- ( \overline{g_u} \mcP e_u, \bftphi_p) \ dt &.
    \end{aligned}
  \end{equation}
  Summing \eqref{eq:err_ep} over all $m$, combining it with
  \eqref{eq:err_eu} and using \eqref{crid_1} and
  \eqref{crid_2} leads to,
  \begin{equation}
    \label{eq:err_eu_ep_I_mn}
    \begin{aligned}
      \int_{I_n} (e_u,\psi_u)& +   \scalar{\bfe_p,\bfpsi_p} \ dt = -\scalar{\bfe_p^- ,\bftphi_p}_{n} -(e_u^-,\widetilde{\phi}_u)_{n}\\
      &  + \scalar{\bfe_p^- ,\bftphi_p}_{n-1}  + (e_u^-,\widetilde{\phi}_u)_{n-1}+ \sum_{m=1}^{M_n} \Big\lbrace  -\scalar{[\bfP]_{m-1,n},\bftphi_{p,(m-1,n)}}  \\
      &+  \int_{I_{m,n}} \scalar{\dotbfe_p ,\bftphi_p} -
      \scalar{ \bfg(\mcP u,\bfp)-\bfg(\mcP U,\bfP), \bftphi_p} \ dt\Big\rbrace -([U]_{n-1},\widetilde{\phi}_{u,(n-1)})
      \\&+ \int_{I_n} (\dot{e}_u,\widetilde{\phi}_u)
      + a(e_u,\widetilde{\phi}_u)  - (f(u,\mcR \bfp)-f(U,\mcR \bfP),\widetilde{\phi}_u) \ dt.
    \end{aligned}
  \end{equation}
  Now, using \eqref{PDE} and
  \eqref{chemistry}, and summing over all $n$,
  \begin{equation}
    \label{eq:err_sum_n_steps}
    \begin{aligned}
      \sum_{n=1}^N \int_{I_n} (e_u,\psi_u)& +   \scalar{\bfe_p,\bfpsi_p} \ dt = (u^0-\Pi_0 u^0,\widetilde{\phi}_{u,0}) + \sum_{n=1}^N \Bigg\lbrace
      \\ \sum_{m=1}^{M_n} \Big\lbrace
      &-\scalar{[\bfP]_{m-1,n},\bftphi_{p,(m-1,n)}}  -  \int_{I_{m,n}} \scalar{\dot{\bfP} ,\bftphi_p} -
      \scalar{ \bfg(\mcP U,\bfP), \bftphi_p} \ dt\Big\rbrace \\
      &-([U]_{n-1},\widetilde{\phi}_{u,(n-1)})
      - \int_{I_n} (\dot{U}_u,\widetilde{\phi}_u)
      + a(U,\widetilde{\phi}_u)  -f(U,\mcR \bfP),\widetilde{\phi}_u) \ dt \Bigg\rbrace.
    \end{aligned}
  \end{equation}
  Addition and subtraction of $\int_{I_n} (f(U,
  \widetilde{\mcR}\bfP),\widetilde{\phi}_u) \ dt$ and $\int_{I_n}
  \scalar{ \bfg(\mcP U_{n-1},\bfP),\bftphi_p } \ dt$, integration by
  parts on $a(U,\widetilde{\phi}_u)$ and the use of Galerkin
  orthogonalities \eqref{Pdg} and \eqref{Udg} to subtract the
  interpolants $\Pi_n \pi_n^u \widetilde{\phi}_u$ and $\bfpi_{m,n}^p
  \bftphi_p$ completes the proof.
\end{proof}

\subsection{The implicit-explicit adjoint problem and error representation formula}
Lemma \ref{lem:errrep_fi} gives an error representation formula that requires numerical solution of the fully implicit adjoint problem. Alternatively, we consider an adjoint that employs the same discretization steps as used for the original model. The exact
implicit-explicit adjoint corresponding to the implicit-explicit numerical scheme reads: For $n=N,\ldots,1$, find $\bfphi_p(t) \in
[\mcL^2(I_{m,n})]^{\Nr}$ on $m=1,\ldots,M_n$ such that
\begin{align}
  \label{adjointP2}
  \int_{I_{m,n}} \scalar{\bfq,-\dotbfphi_p}  - \scalar{\bfq,\overline{\bfg_p}\hspace{0pt}^* \bfphi_p + \mcR^* \overline{\bff_p}\hspace{0pt}^*  \phi_{u,n}} \ dt &= \int_{I_{m,n}} \scalar{\bfq,\bfpsi_p} \ dt.
\end{align}
Note the use of $\phi_{u,n}$, which is known from the time interval
$I_{n+1}$.  Then find $\phi_u(x,t) \in \mcL^2(I_n,\mcH^1(\Omega))$ such that
\begin{align}
  \label{adjointU2}
  \int_{I_n} (v,-\dotphi_u) + a(v,\phi_u)  - (v,\overline{f_u} \hspace{0pt}^* \phi_u + \mcP^* \overline{g_u}\hspace{0pt}^* \bfphi_p) \ dt &= \int_{I_n} (v,\psi_u) \ dt.
\end{align}
The corresponding iterative method follows.

\begin{theorem}[Error representation formula]\label{lem:errrep}
  Let the quantity of interest be the linear functional $m(u,\bfp)$
  be defined by the functions $\psi_u\in \mcL^2(\Omega)$ and
  $\bfpsi_p \in [\mcL^2(\Omega)]^{\Nr}$ such that $m(u,\bfp) = \int_0^T
  (u,\psi_u) + \scalar{ \bfp,\bfpsi_p } \ dt$. The error representation
  formula for the error $E(U,\bfP) = \abs{ m(u,\bfp) - m(U,\bfP)}$ now
  reads
  \begin{align}
    \label{errorformula2}
    E(U,\bfP)
    &=  \Big| \sum_{n=1}^N \int_{I_n} (e_u,\psi_u) + \scalar{\bfe_p,\bfpsi_p} \ dt \Big| \\
    &= \Big| I + \sum_{n=1}^N ( II_n + III_n + IV_n + V_n + V\!I_n) \Big|,
  \end{align}
  where
  \begin{align}
    \label{errorI}
    I &=  (u^0-\Pi_0 u^0,\phi_{u,0}), \\
    \label{errorII}
    II_n &= \sum_{K\in \mcT_n^h} \int_{I_n} (R_{u,K}(U,\bfP), \phi_u - \Pi_n \pi_n^u \phi_u)_K \ dt, \\
    \label{errorIII}
    III_n &= \sum_{m=1}^{M_n} \int_{I_{m,n}} \scalar{\bfR_p(\bfP), \bfphi_p - \bfpi_{m,n}^p \bfphi_p} \ dt, \\
    \label{errorIV}
    IV_n &= \int_{I_n} (f(U,\mcR \bfP) - f(U,\widetilde{\mcR}\bfP), \phi_u) \ dt, \\
    \label{errorV}
    V_n &=  \int_{I_n} \scalar{ \bfg(\mcP U,\bfP) - \bfg(\mcP U_{n-1},\bfP), \bfphi_p } \ dt, \\
    \label{errorVI}
    V\!I_n &= \int_{I_n} (\overline{\bff_p} \mcR \bfe_p, \phi_u - \phi_{u,n} ) \ dt.
  \end{align}
\end{theorem}

The first five terms are similar to the error representation in
Lemma \ref{lem:errrep_fi}. The last term is the contribution of the
transfer error from the ODEs to the PDE through $f$, weighted by the
effect of the splitting of the adjoint equations.

\begin{proof}
  The proof proceeds as in the case of Theorem
  \ref{lem:errrep_fi}. However we have to account for the
  implicit-explicit solve of the adjoint, which shows up in the analogue of
  \eqref{eq:err_ep} as,
  \begin{equation}
    \label{eq:err_ep_ei}
    \begin{aligned}
      \int_{I_{m,n}} \scalar{\bfe_p,\bfpsi_p} \ dt &= -\scalar{\bfe_p^- ,\bfphi_p}_{m,n} + \scalar{\bfe_p^- ,\bfphi_p}_{m-1,n}  - \scalar{[\bfP]_{m-1,n},\bfphi_{p,(m-1,n)}}\\
      & \phantom{a} \quad + \int_{I_{m,n}} \scalar{\dotbfe_p ,\bfphi_p}  - \scalar{\overline{\bfg_p} \bfe_p , \bfphi_p} - \scalar{ \overline{\bff_p} \mcR \bfe_p , \phi_{u,n}} \ dt.
    \end{aligned}
  \end{equation}
  The presence of the term involving $\phi_{u,n}$ prevents use of the fact that $\bfp$ and $u$ satisfy \eqref{PDE} and
  \eqref{chemistry} . Thus, we add and subtract $\int_{I_{m,n}} \scalar{
    \overline{\bff_p} \mcR \bfe_p , \phi_{u}} \ dt$ to
  \eqref{eq:err_ep_ei}. The rest of the proof follows as before. The
  addition and subtraction of the additional term leads to the sixth
  term in the error representation formula.
\end{proof}

\subsection{Error indicators}

\begin{theorem}[Error indicators]\label{lem:errind}
  By decomposing the contributions to the discretization error of the PDE into a spatial and
  temporal parts as $II_n = II_n^x + II_n^t$, we can distinguish the
  contributions from the spatial discretization for the PDE as $E^x$, the temporal discretization
  of the PDE by $E^t$ and the ODE discretization contribution as $E^s$, so
  \begin{equation*}
  E(U,\bfP) \leq  E^x+E^t+E^s ,
 \end{equation*}
  where
  \begin{align}
    \label{Ex}
    E^x
    &= I + \sum_{n=1}^N (II_n^x + IV_n),\\
    \label{Et}
    E^t
    &= \sum_{n=1}^N  (II_n^t + V_n),\\
    \label{Eo}
    E^s
    &= \sum_{n=1}^N III_n.
  \end{align}
  Furthermore, these
  indicators can be approximated by the computable indicators $\AE^x$,
  $\AE^t$ and $\AE^s$ as
  \begin{align}
    \label{aex}
    \AE^x &= \sum_{K \in \mcT_0^h} \abs{ I_K} + \sum_{n=1}^N \sum_{K \in \mcT_n^h} \abs{ II_{n,K}^x }  + \sum_{n=1}^N \abs{ IV_n }, \\
    \label{aet}
    \AE^t &= \sum_{n=1}^N ( \abs{ II_n^t} + \abs{ V_n} ), \\
    \label{aeo}
    \AE^s &= \sum_{n=1}^N \sum_{m=1}^{M_n} \abs{ III_{m,n} },
  \end{align}
  where the subscripts indicate restriction. For these to be computable
  we replace $\phi_u$ and $\bfphi_p$ by the discrete approximations
  $\Phi_u$ and $\bfPhi_p$ respectively. This leads to additional terms in the representation that cannot be estimated but are relatively small in the limit of discretization refinement.
\end{theorem}

\begin{proof}
  We have to account for the effect of replacing the exact adjoints $\phi_u$ and $\bfphi_p$
  with the approximations $\Phi_u \in \mcW^{r_u+1}_n$ and $\bfPhi_p \in
  \mcQ_{m,n}^{r_p+1}$ obtained by finite element discretizations similar
 to \eqref{Udg} and \eqref{Pdg}, but with $\bfq\in
  \mcQ_{m,n}^{r_p}$ and $v\in \mcW^{r_u}_n$. The effect on $E(U,\bfP)$ of replacing $\phi_u$ with
  $\Phi_u$ in   \eqref{errorI} -- \eqref{errorVI} leads to additional terms with weights of the type   $\phi_u-\Phi_u$. However, these additional terms are higher order. The same holds for replacing   $\bfphi_p$ with $\bfPhi_p$. Since these errors are negligible we do
  not indicate this change in the terms $I$ -- $V\!I$ below.

  To have separate indicators for the spatial and the temporal contributions from discretization of   the PDE and the contribution for the discretization of the ODEs, we  decompose the error
  $E(U,\bfP)$ into parts corresponding to these contributions. To distinguish
  the different contributions, we note that term $II$ \eqref{errorII}
  contributes both to errors in space as well as in time. To deal with this,
  we write
  \begin{align}
    \Phi_u - \Pi_n \pi_n^u \Phi_u = \Phi_u - \Pi_n \Phi_n + \Pi_n \Phi_n -
    \Pi_n \pi_n^u \Phi_u,
  \end{align}
  which gives $II_n = II_n^x + II_n^t$, where
  \begin{align}
    \label{IInx}
    II_n^x &= \sum_{K \in \mcT_n^h} \int_{I_n} (R_{u,K}(U,\bfP),\Phi_u - \Pi_n \Phi_u)_K \ dt, \\
    \label{IInt}
    II_n^t &= \sum_{K \in \mcT_n^h} \int_{I_n} (R_{u,K}(U,\bfP),\Pi_n \Phi_u - \pi_n^u \Pi_n \Phi_u)_K \ dt,
  \end{align}
  where $R_{u,K}$ is the algebraic equivalent to \eqref{resu}, since
  $\Pi_n \Phi_u \in \mcV_n^h$.

Term $IV_n$ involve the exact quantity $\mcR$, which is not
computable. We compute an approximation of this term by sampling at a large number of points
to form $\mcR$ and then choose only a subset of these samples to form $\widetilde{\mcR}$.
\end{proof}

Finally, the Term $V\! I_n$ involving $\bfe_p$, which arises from the choice of a computationally-tractable adjoint problem, is not directly computable.  In some cases, such terms can be approximated at the cost of solving auxiliary adjoint
problems~\cite{CET09}. Alternatively, a rigorous mathematical analysis showing such
terms are small relative to the computable terms in the error estimate is
carried out in~\cite{CEGT12a}.  Intuitively $V\! I_n$ is small in the limit of discretization refinement because involves a product of the errors $\bfe_p$ and
$\phi_u-\phi_{u,n}$. On the other hand, if $V\! I_n$ is large, we
expect the other terms in the error estimate to be large as
well. Hence, while the error estimate may not capture the true error
accurately in this case, it is still be reliable in the sense of
indicating a large error in the numerical solution.

\section{Details of the discretization and the adaptive algorithm}

\subsection{Discretization details}

For discretization in space of the PDE, the partition $\mcT_n^h$ of $\Omega$ consists of hexahedral elements
with trilinear basis functions for the primal PDE problem and
triquadratic basis functions for its adjoint. The capability for spatial mesh adaptivity, including refinement and derefinement, is provided by the use of hanging nodes, where at
most one hanging node per edge or face is allowed. Conformity of the
basis at the hanging nodes is obtained by interpolation using the
neighboring nodes, see \cite{Ainsworth}. Handling of such shape
regular but non-uniform meshes is alleviated by an octree-based data
structure \cite{frisken2005simple}.

For the stochastic discretization of the microscale ODEs, the Voronoi cells $\omega_j$ are generated by $\No$ uniformly distributed random points in $\Omega$. Given these points, the
corresponding Voronoi tessellation is generated by calling the
\textit{Voro++} library \cite{Rycroft}.  Since this tessellation is
generally \emph{not} aligned with $\mcT_n^h$ (as can be seen in Fig.
\ref{fig:vorocube}) a map from each quadrature point in each
$\mcT_n^h$ to the corresponding Voronoi cell is constructed. This
procedure requires searching the Voronoi diagram, but the library
supplies efficient routines for doing this. Moreover, $\mcT_n^h$ only
varies when the mesh is changed, which happens fairly infrequently in the block adaptivity approach that we use (described in the next section). We
sample by using 10 random points in each $\omega_j$ to compute $\mcR
P$ in Term IV (cf.\ \eqref{errorIV}) and sample using 1 point to
compute \eqref{MCrecovery1} or \eqref{MCrecovery2}.

Computing the projection $\mcP$ is performed using the same data structures created for $\mcR$: We sample 10 points in each $\omega_j$, and for
each of these points, we find the corresponding element in $\mcT_n^h$
and compute the function value at that point. Then we
average over these function values to get the projection over
$\omega_j$.

The temporal discretization of the primal and adjoint PDE problem is
performed using a dG(0) and a cG(1) respectively. For the ODEs, a
dG(1) method is used for the primal and a dG(2) method for its adjoint.
The ODEs are initially solved on the same temporal discretization as
the PDE. However, the ODEs are allowed to have individual time steps
in each slab $K\times I_n$ (see Section \ref{sect:fem}). To determine
these time steps, as well as the discretization of the PDE, we use
blockwise adaptivity \cite{blockadap}, which is a more efficient
procedure than a standard compute -- estimate -- mark -- refine
strategy. This is described in the following section. Finally, the adjoint time discretization is set to be equal to the temporal discretization for the forward problem.

To be able to represent data on the different grids, the octree data
structure facilitates the projection and interpolation routines that
are necessary. To alleviate the matrix allocations,
assembly of the matrices and vectors as well as solving the linear
systems involved we take use of the \textit{PETSc} library
\cite{petsc}. We use its conjugated gradient method with the
incomplete LU factorization as preconditioner. A Newton method is used
for performing the nonlinear iterations for the primal PDE and
ODEs.

Finally, despite the higher order method used for the adjoints, the costs
can be compared with those for the forward problem due to the fact
that they are linear.

\subsection{Blockwise adaptive error control}
\label{sect:adaptivealg}

To determine spatial and temporal refinement, the adaptive algorithm
called blockwise adaptivity is used. This procedure, introduced by
Carey et.~al.~in \cite{blockadap}, is built upon the creation of a
sequence of blocks of meshes. The meshes that constitute the block are
carefully created to guarantee that the numerical method is accurate
in terms of the goal functional as well as being efficient in the
sense that they are as coarse as possible.

The blocks form the partition in time by $0 = T_0 < T_1 < \cdots < T_B
= T$. In each block, both the spatial and temporal meshes are fixed
but possibly non-uniform. Given initial primal and adjoint solutions on a
coarse discretization, the algorithm is based on an absolute tolerance
$\ATOL$ for the total error, the Principle of Equidistribution and a
strategy for computing blocks. The Principle of Equidistribution says
that an optimal mesh for controlling $\ATOL$ is obtained when the
element contributions are approximately equal, see e.g.~Eriksson
et.~al.~\cite{actaintro}.

The strategy may be based on a criteria, for example the maximum size
of a mesh due to limitations in computer memory. This strategy, known
as a memory-bound strategy, is employed in this paper. Other
strategies include forming blocks by considering changes in the
topology of the meshes. For the memory-bound strategy we place upper
bounds on the number of degrees of freedom, such that there are
\begin{itemize} 
\item a maximum number of spatial elements $\xMAX$,
\item a maximum number of PDE time intervals $\tMAX$,
\item a maximum number of ODE time intervals $\sMAX$.
\end{itemize}
Each of this restrictions may define end time of blocks. For
completeness we briefly present this method, but is thoroughly
described in \cite{blockadap}.

Assuming that in a mesh there are $N_x$ space elements, $N$ macro time
intervals and $M_n$ subintervals in each macro interval, there is a
total of $N_x N M_n$ space-time subslabs in the initial
discretization.  The Principle of Equidistribution determines an
approximate local error tolerance on each subslab of
\begin{align}
  \LATOL = \frac{\ATOL}{3N_x N M_n},
\end{align}
since there are three different contributions to the error which
should be equal: The spatial and temporal part of the PDE and the ODE.

The next step is to predict meshes such that each element contribution
is approximately $\LATOL$. For example, standard a priori analysis may
show that the error $E^s|_{I_{m,n}}$ on a subinterval of size $\Delta
s$ is of the order of $(\Delta s)^{r}$ for some $r$. Due to the
Principle of Equidistribution, the desired error should be
$\LATOL$. If this requires a subinterval of size $\Delta s_{new}$ we
have by proportionality
\begin{align}
  \LATOL \approx E^s|_{I_{m,n}} \times \left( \frac{\Delta s_{new}}{\Delta s},
  \right)^{r},
\end{align}
which results in that the approximate subinterval length can be
determined as
\begin{align}
  \Delta s_{\new} \approx \Delta s \times  \left( \frac{\LATOL}{E^s|_{I_{m,n}}} \right)^{1/r},
\end{align}
or as a number of new subintervals by
\begin{align}
  \label{timepredict}
  \frac{\Delta s}{\Delta s_{\new}} = \left( \frac{E^s|_{I_{m,n}}}{\LATOL} \right)^{1/r}.
\end{align}
It should be noted that if the original discretization is finer than
needed, then $\Delta s_{\new} > \Delta s$, which suggests
coarsening. Moreover, $\Delta s_{\new}$ is in fact a computable
quantity, since $E^s|_{I_{m,n}} \approx \AE^s|_{I_{m,n}}$
(cf.~equations \eqref{Eo} and \eqref{aeo}). For the space and
discretizations of the PDE, the procedure is similar and uses the
approximations $\AE^x$ and $\AE^t$ defined in \eqref{aex} and
\eqref{aet}.

Using the memory bound strategy, we determine that the predicted
number of elements that are of interest. Thus, besides
\eqref{timepredict} we also need to predict the number of elements in
space and time for the PDE. The latter is done using a similar
quotient as \eqref{timepredict}, but for the former, there is a
scaling factor with the dimension $d=3$ as
\begin{align}
  \label{npredict}
  \left(\frac{h}{h_{new}}\right)^d \approx \left( \frac{\AE^x|_{n,K}}{\LATOL} \right)^{d/r},
\end{align}
where $r$ is the order of convergence in space.

Using the predicted number of spatial elements \eqref{npredict}, one
can create a block by starting at $T_0 = 0$, loop over time and add up
the total number of predicted number of elements until the sum is
about equal to $\xMAX$ or $\tMAX$. Then this time defines $T_1$. If
the maximum predicted number of ODE steps over all spatial elements
exceeds $\sMAX$, that PDE time interval is to be refined, thus
increasing the predicted number of PDE time intervals. The second
block starts with $T_1$ and is formed analogously.

Derefinement is restricted to the resolution of the initial coarse
mesh, both is space and time. In addition to the three restriction
criteria above, a sufficiently coarse predicted mesh may also define
the end of a block. This can be determined by a parameter $\theta$
defined as
\begin{align}
  \label{eq:theta}
  \theta = \frac{\#\text{elements in current block}}{\#\text{elements predicted for the next block}}.
\end{align}
We coarsen the mesh when $\theta > 10$.

We note there can be significant contributions to the error if the meshes on adjacent blocks are sufficiently different \cite{blockadap}.  We do not include expressions estimating these contributions in the a posteriori error estimate. Rather, we reduce the impact of such contributions by using the maximum error estimate in the final time interval of the previous block to  determine the refinements of the first time interval on the next block.

Below, we use one iteration of the block adaptive algorithm, so  the PDE-ODE system is solved once
on the initial coarse mesh. This initial mesh is then refined in space and time
according to the strategy described above with $\ATOL\approx 10\%$ of
the approximate total error estimate $\AE^x + \AE^t + \AE^s$ for the
initial mesh.

\section{A numerical example}

We consider a problem from
electrocardiography  known as the monodomain
model (cf.~\cite{collifranzone}), which is \eqref{PDE} where
$u$ is the transmembrane potential, i.e.~the potential difference over
the cell membrane and $f$ is a current of charged ions. We assume
$\Omega=[0,1]^3$, $T=400$ ms, and a constant conductivity $\epsilon=0.1$. The
initial value $u^0$ is set such that there is an excited region close
to the origin, which decreases to a resting value using smoothed
Heaviside functions $G_\delta$ as,
\begin{align}
  u^0(r) = v_1(1-G_\delta(r))+v_0G_\delta(r),
\end{align}
where $r=\abs{x}-r_0$ and,
\begin{align}
  G_\delta(x) =
  \begin{cases}
    1, & x > \delta, \\
    \displaystyle \frac{1}{2} \left(
    1+\frac{x}{\delta} + \frac{1}{\pi} \sin \frac{\pi x}{\delta}
    \right), &  \abs{x} \leq \delta, \\
    0, &  x < -\delta,
  \end{cases}
\end{align}
with $v_0=-84.624$ mV, $v_1=20$ mV, $r_0=0.5$ and $\delta=0.2$.

The monodomain model is coupled to an ODE model called the
Beeler-Reuter model, a common model for describing mammalian
ventricular action potentials, i.e., the voltage response of a
specific type of heart muscle cells. The model is described in full in
Beeler and Reuter's original paper \cite{beelerreuter} and we review
it briefly here for completeness.

The model contains seven ODE variables, where six describe gates
governing the inflow and outflow of certain ions through the cell
membrane. These gating variables are modeled on the assumptions by
Hodgkin and Huxley saying that the openness depends on the proportion
of already open channels, $\bfp_i$, and the proportion of closed
channels, $1-\bfp_i$. Thus, $\bfp_i\in [0,1]$,
$i=1,\ldots,6$. Moreover, the level of openness is in turn depending
on the potential by rate functions $\alpha_i(u)$ and $\beta_i(u)$ to
form
\begin{align}
  \dot{\bfp}_i &= \alpha_i(u) (1-\bfp_i) - \beta_i(u) \bfp_i, \quad i = 1,\ldots,6.
\end{align}
Three of the gating variables are fast and thus sensitive to changes
in the potential, the other three are slower. See the seminal paper by
Hodgkin and Huxley \cite{hodhux} for a thorough analysis of the
modeling of the gating variables. Beeler and Reuter uses the following
general expression for the rate functions
\begin{align}
  \frac{c_1 \exp(c_2(u+c_3)) + c_4(u+c_5)}{\exp(c_6(u+c_3)) + c_7},
\end{align}
and the values for the constants $c_k$, $k=1,\ldots,7$, can be found
in \cite{beelerreuter}.

The seventh ODE variable models the concentration of calcium ions
inside the cell, and is governed by the ODE
\begin{align}
  \dot{\bfp}_7 &= -a_0 j_S(u,\bfp) - b_0(a_0-\bfp_7),
\end{align}
where $a_0$ and $b_0$ are constants and $j_S$ is a current depending
on the potential $u$ and the gating variables. The coupling to the PDE
is through the ion current $f=j_{ion}$ which is sum of different
currents,
\begin{align}
  j_{ion}(u, \bfp) = -(j_{Na}(\mcP u,\bfp) + j_S(\mcP u, \bfp) + j_{X}(\mcP u,\bfp) + j_{K}(\mcP u)).
\end{align}
Similar to $j_S$, the currents $j_{Na}$ and $j_{X}$ depend on the
potential $u$ and the gating variables. The current $j_{K}$ depend
only on the potential $u$.

We are interested in controlling the average error in the PDE solution, so
we set $\psi_u=1$ and $\bfpsi_p=\bfzero$.

For the initial temporal discretization of the PDE, a time step of
$\Delta t = 0.1$ is chosen up to $t=10$, after which point we use
$\Delta t=1$. The total initial total number of time steps is thus
490. The motivation for this initial non-uniform discretization is
that it is known a priori that the dynamics is rapid in the
beginning of the front. We use $L_n=2$ for all $n$.

To form the initial spatial mesh, $\Omega$ is uniformly subdivided
into $16^3=4096$ elements. The adjoint PDE problem is solved on this
mesh resolution for all times, even though the primal mesh is
refined. Moreover we set $\No=75000$.

The parameters for blockwise adaptivity are set to be $\xMAX = 75000$
(i.e.\ equal to $\No$), $\tMAX=1000$ and $\sMAX=20$.

\subsection{Results}
\label{sect:results}

Fig. \ref{fig:pdetimes1-3dplot} illustrates $U$ and $\Phi_u$ at
$t=0.5$ ms on the refined mesh. It is clear how the adjoint is large at
the front. For later times, after wave has passed over the domain, the
solution have little variation in space as shown in Fig.
\ref{fig:pdetimes3dplot-later}. The adjoint also has little variation in
space.

\begin{figure}[htb]
  \begin{center}
    \includegraphics[scale=0.18]{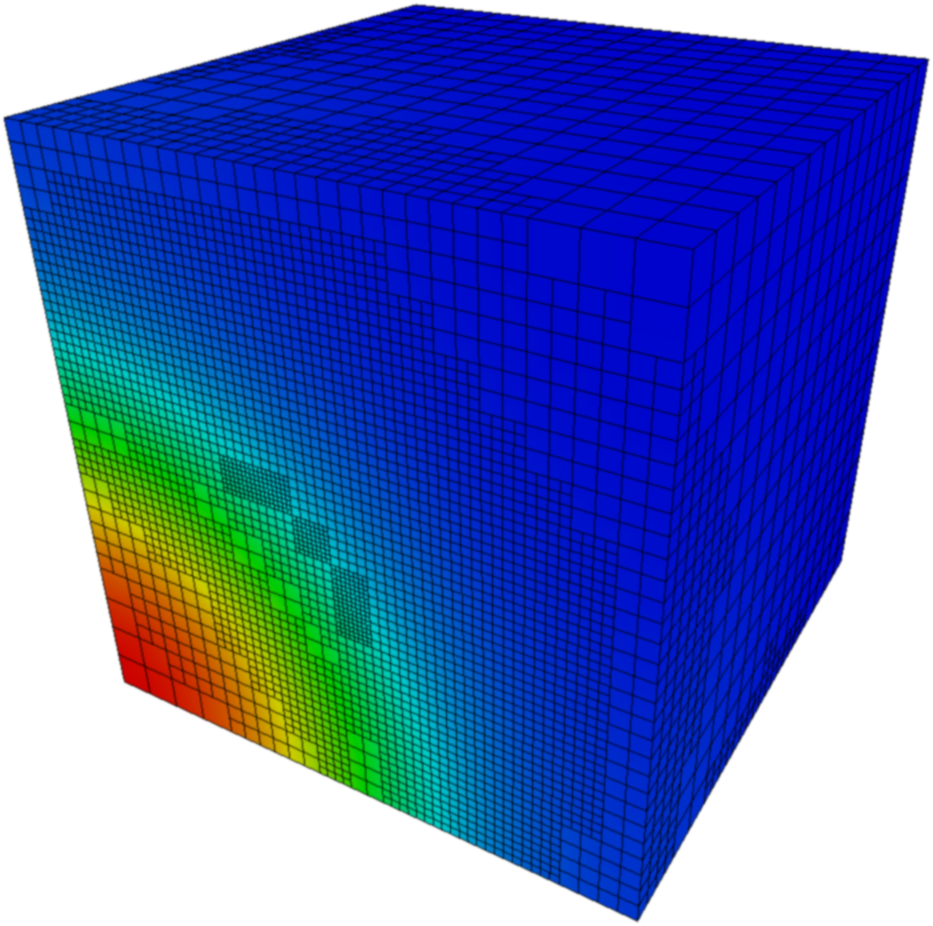} \qquad
    \includegraphics[scale=0.18]{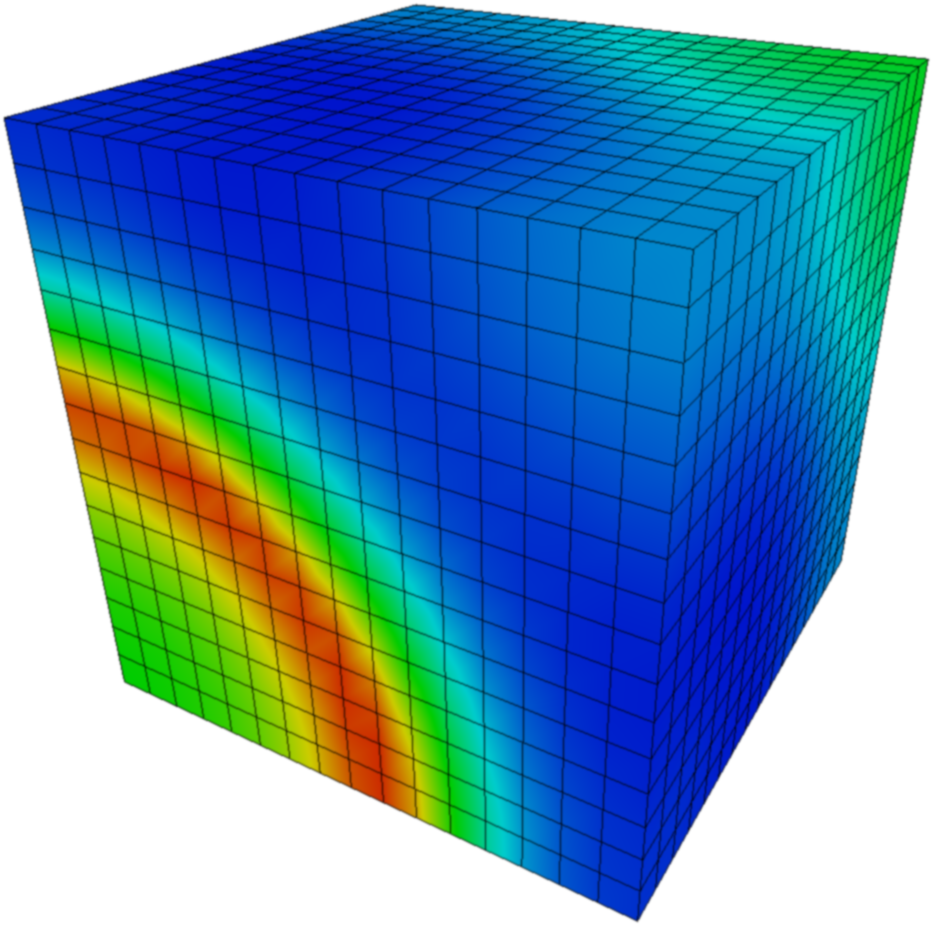}
    \caption{The primal (left) and adjoint (right) PDE solutions at time
      0.5 ms (block no.~1).}
    \label{fig:pdetimes1-3dplot}
  \end{center}
\end{figure}

\begin{figure}[htb]
  \begin{center}
    \includegraphics[scale=0.18]{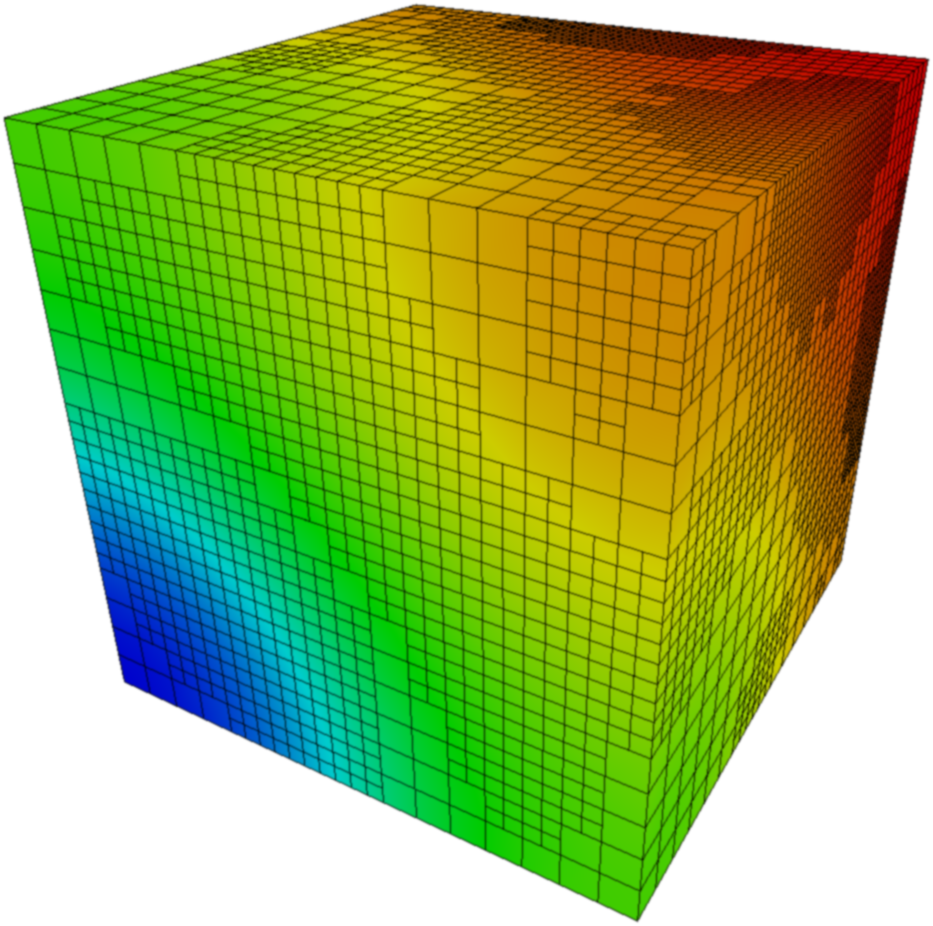} \qquad
    \includegraphics[scale=0.18]{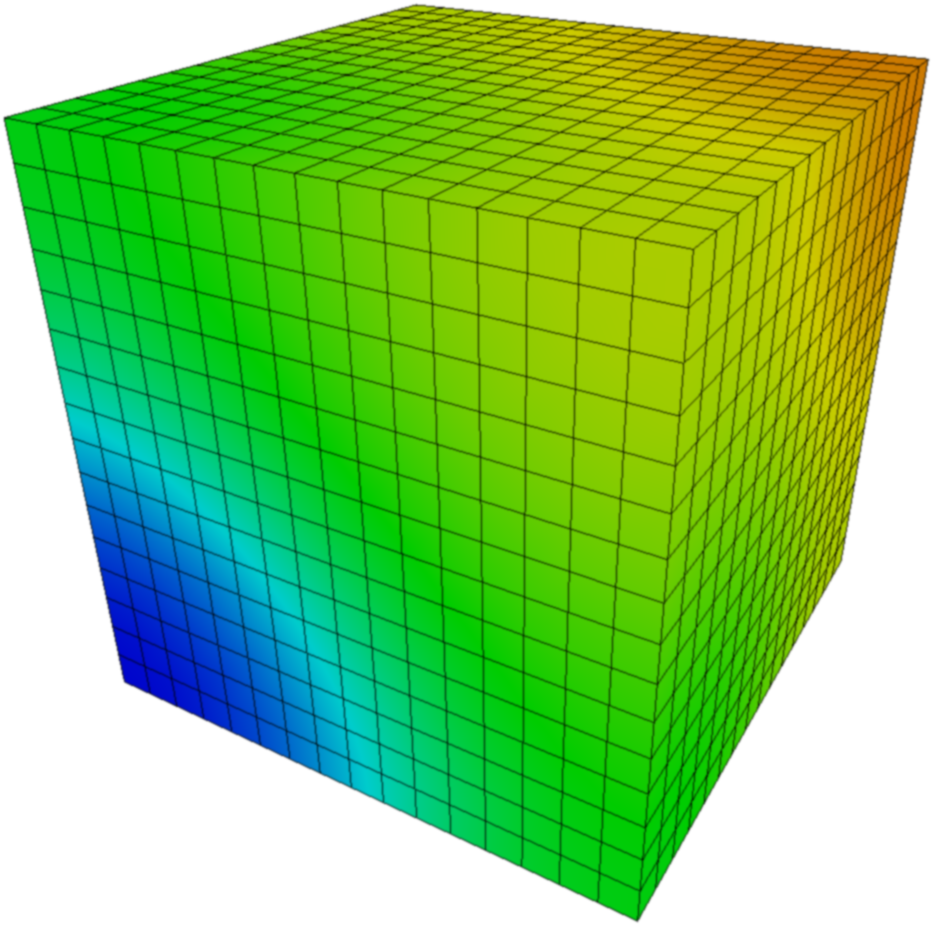}
    \caption{The primal (left) and adjoint (right) PDE solutions at time
      200 ms (block no.~5). The range of the primal solution is
      $-11.76$ to $-11.70$ mV and for the adjoint solution $1.405\cdot 10^4$ to
      $1.406\cdot 10^4$. }
    \label{fig:pdetimes3dplot-later}
  \end{center}
\end{figure}

The dynamics of the PDE and its adjoint are clearly visible in Fig.
\ref{fig:primaldualPDEplot}, which illustrates the solutions in the
center of $\Omega$, i.e.\ at $(0.5,0.5,0.5)$. The left graph shows
that there is a very steep gradient of $U$ in the first 5 ms of the
simulation after which the solution then decays fairly smoothly back
to its initial and resting value after about 300 ms. The adjoint solution
$\Phi_u$ in the right plot of Fig.~\ref{fig:primaldualPDEplot} has a
large peak around 240 ms. This may seem surprising at first, but its
explanation can be found by looking at the dynamics of the coupled
ODEs in Figures \ref{fig:ODEprimaldualfast},
\ref{fig:ODEprimaldualcalcium} and \ref{fig:ODEprimaldualslow}.

\begin{figure}[tb]
  \begin{center}
    \includegraphics[height=4.5cm]{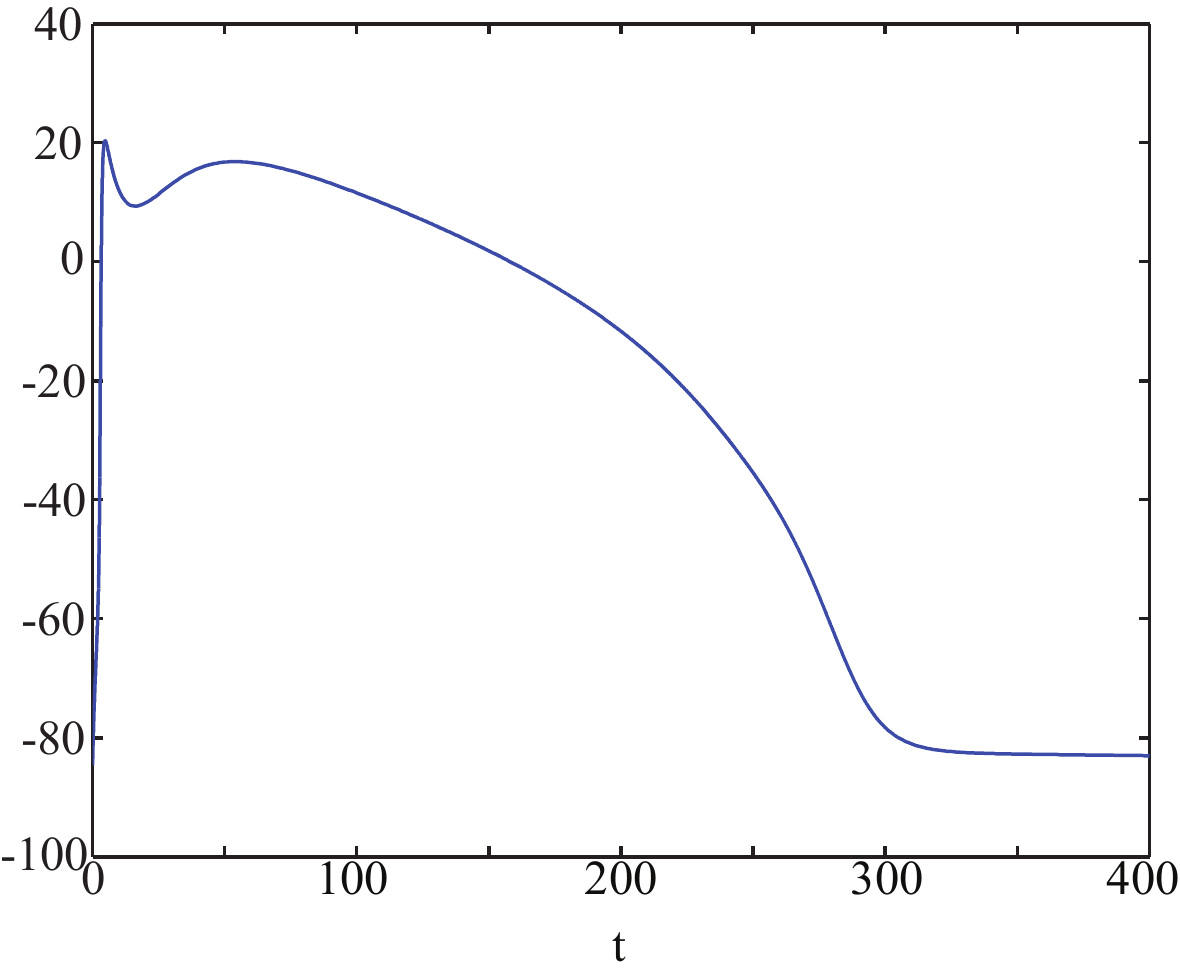}\qquad
    \includegraphics[height=4.5cm]{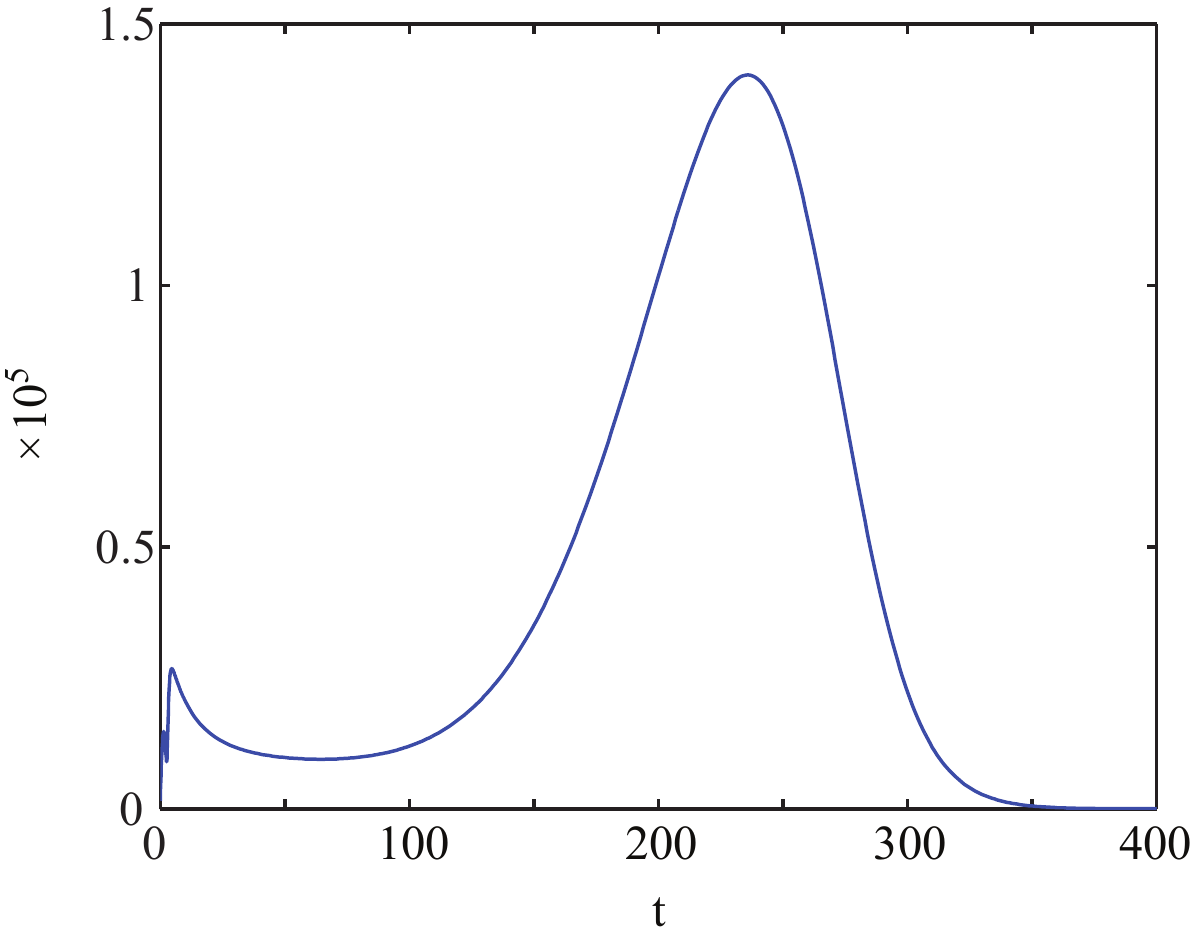}
    \caption{Solutions of the primal (left) and adjoint (right) PDE
      measured in the center of $\Omega$. }
    \label{fig:primaldualPDEplot}
  \end{center}
\end{figure}

\begin{figure}
  \begin{center}
    \includegraphics[height=4.5cm]{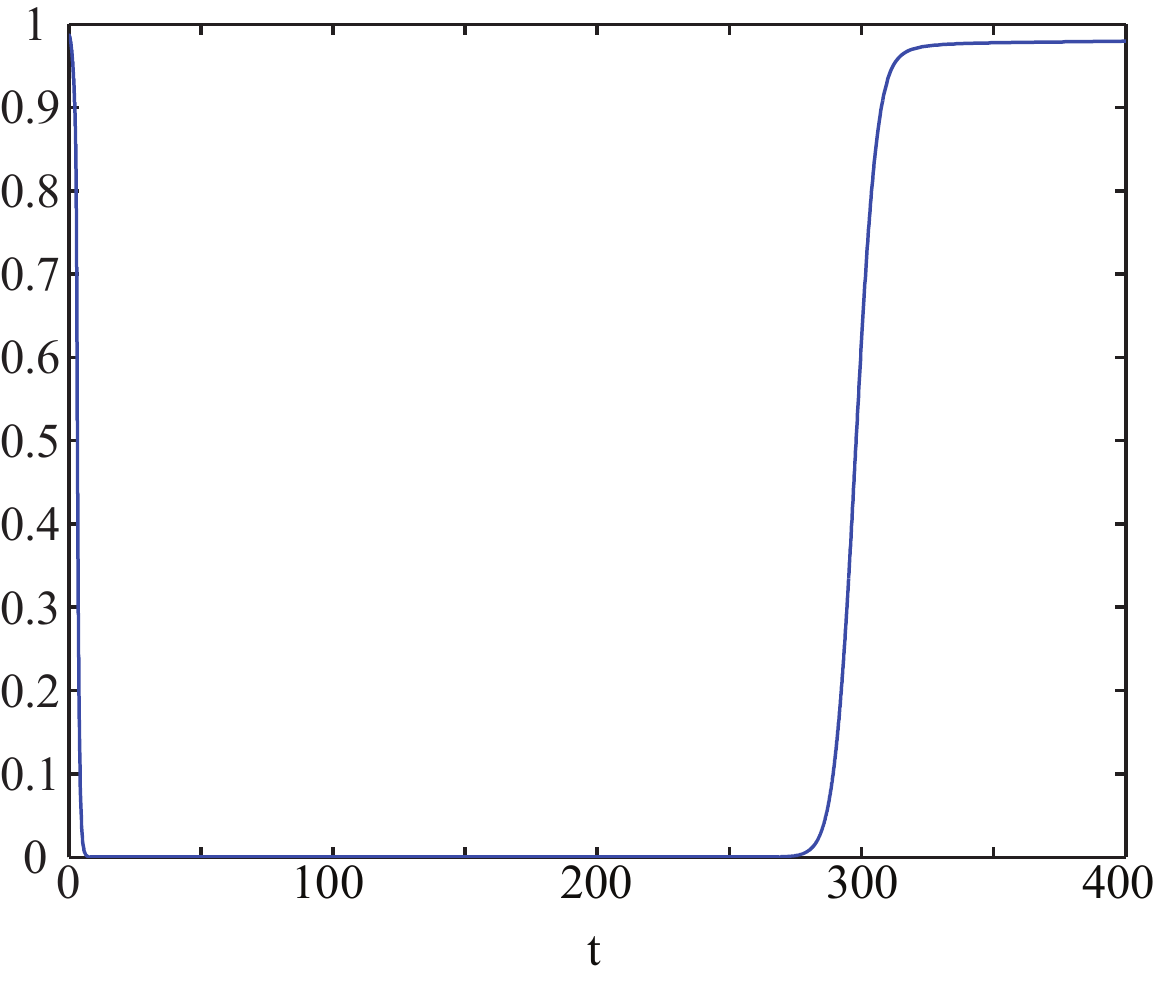} \qquad
    \includegraphics[height=4.5cm]{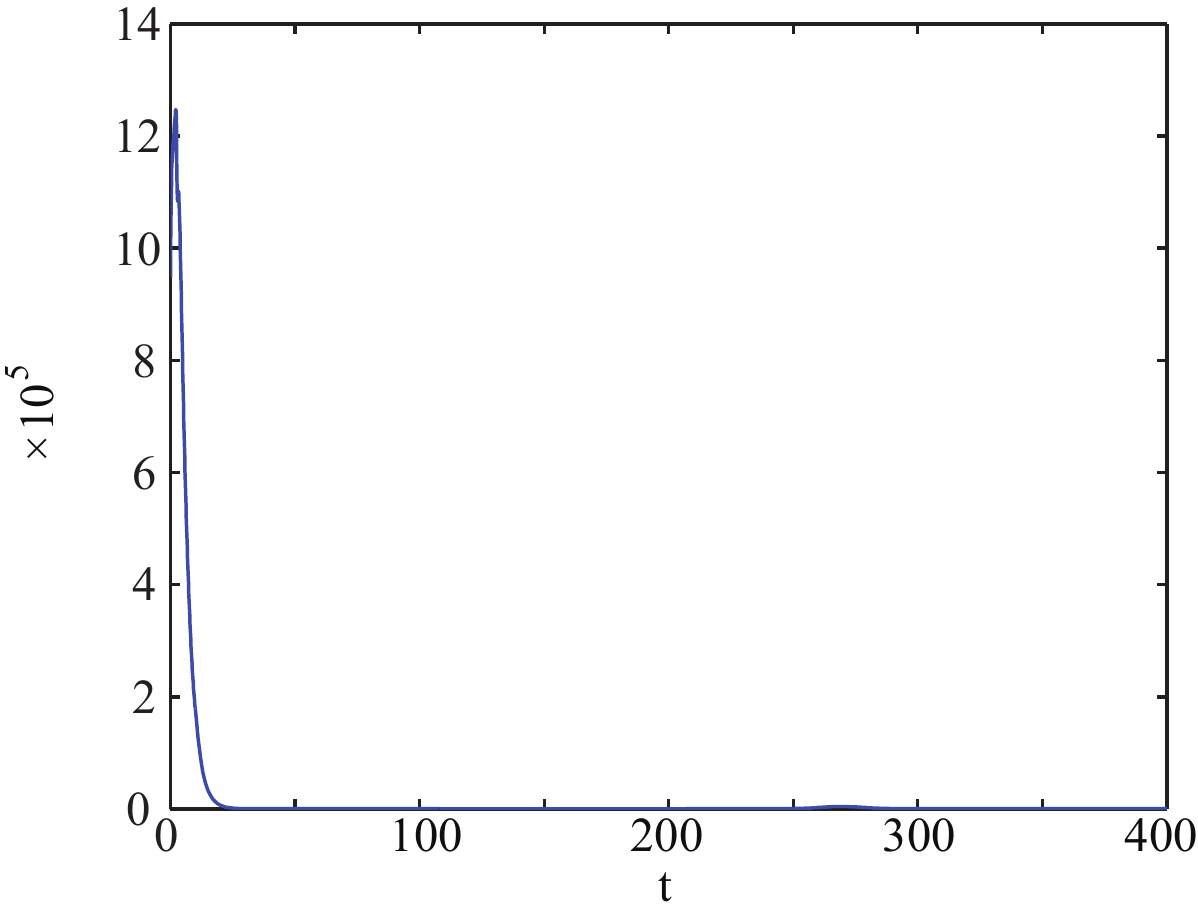}\\
    \includegraphics[height=4.5cm]{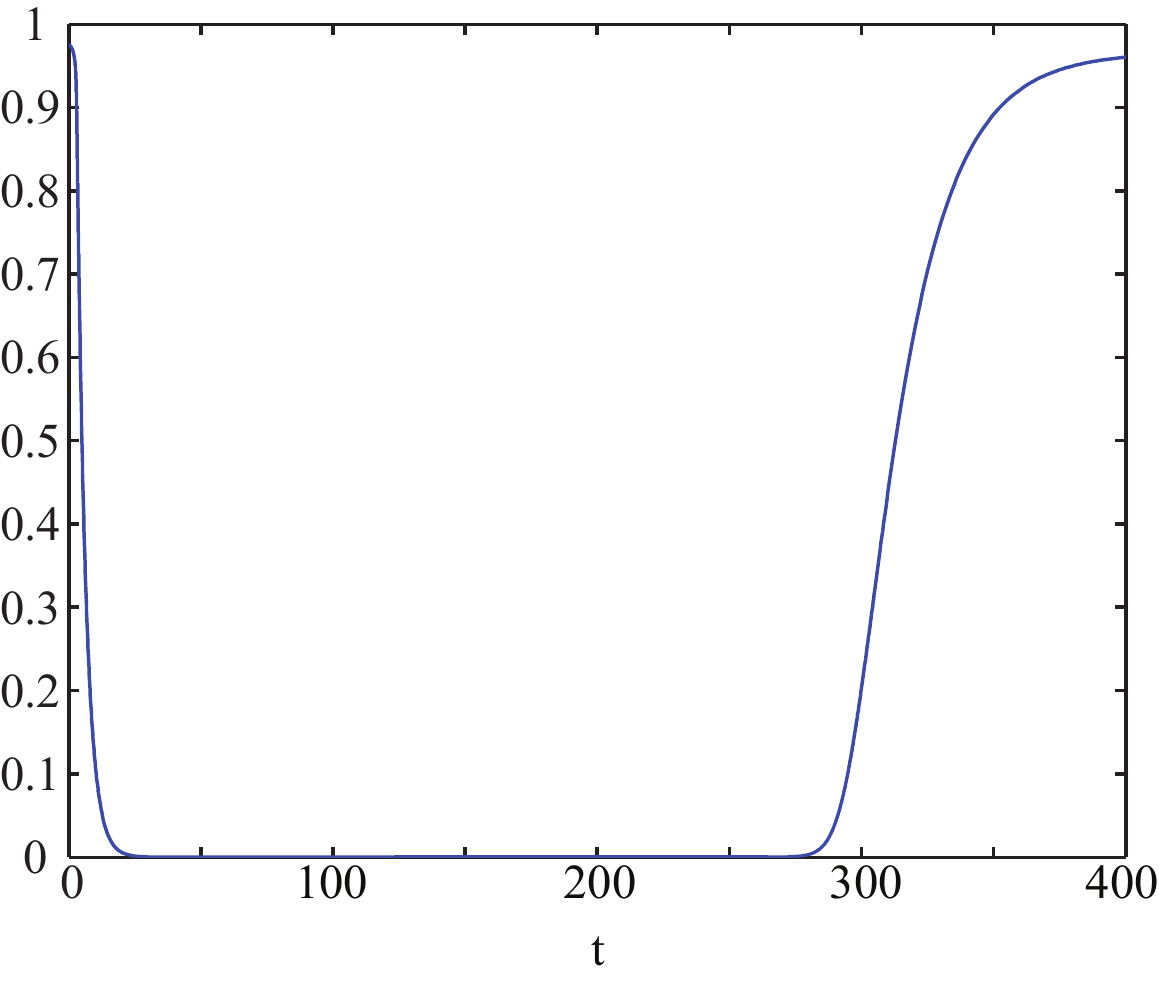} \qquad
    \includegraphics[height=4.5cm]{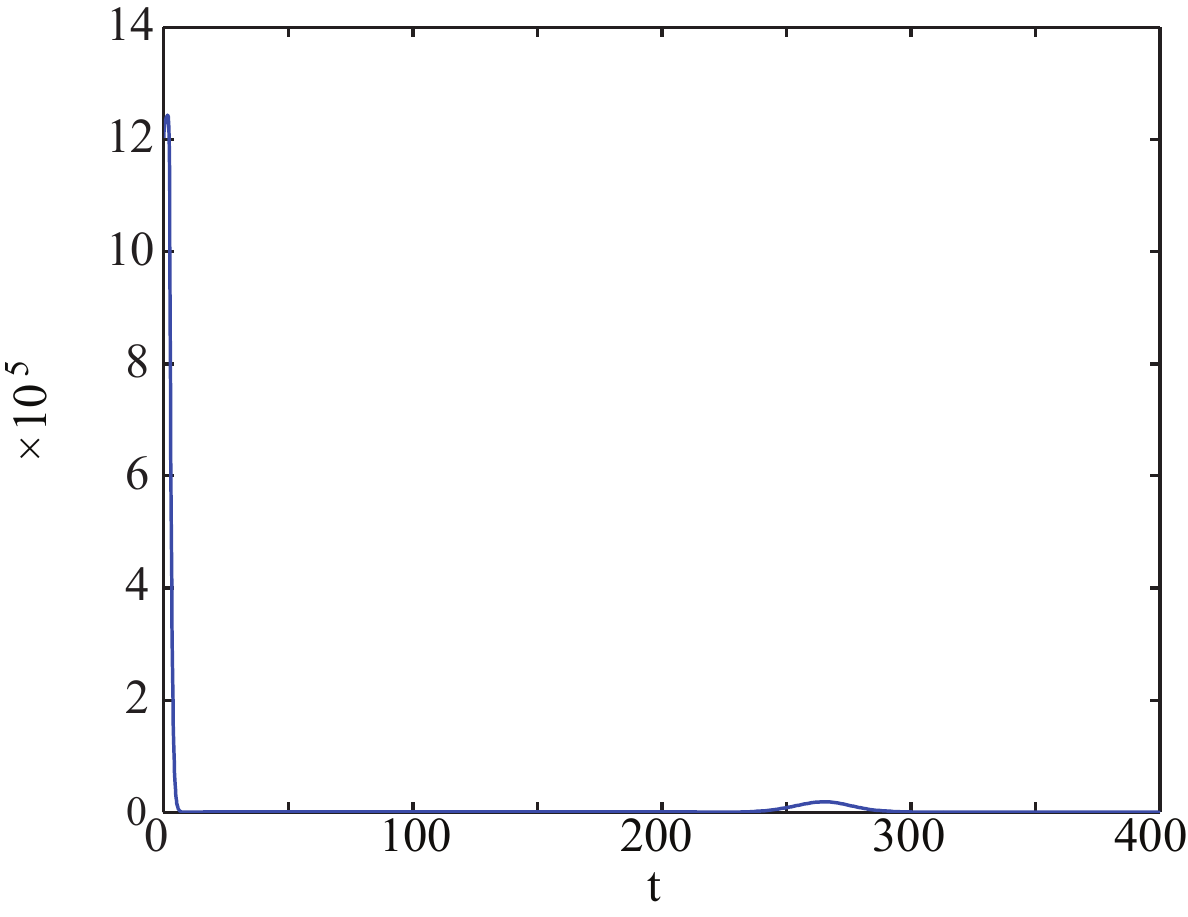}\\
    \includegraphics[height=4.5cm]{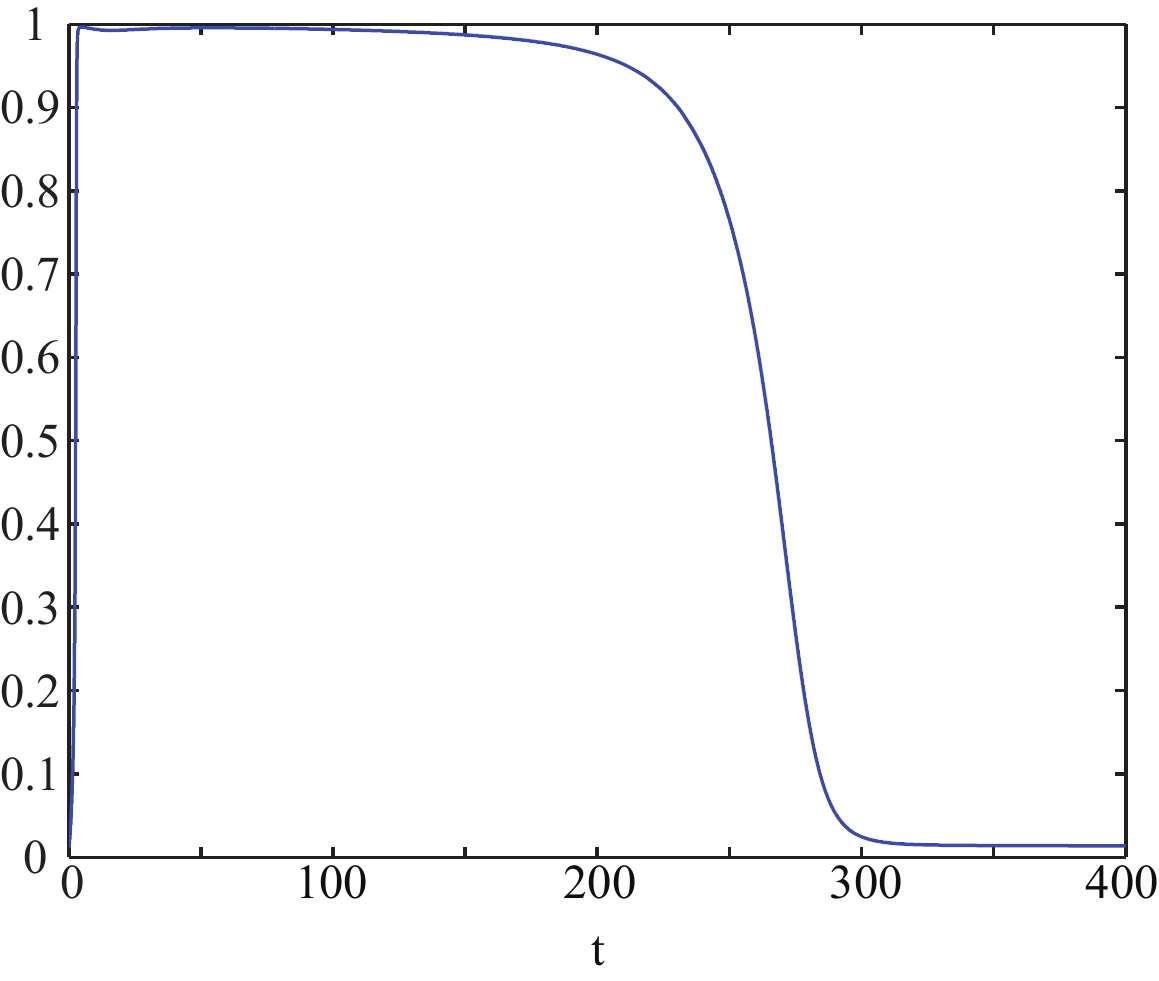} \qquad
    \includegraphics[height=4.5cm]{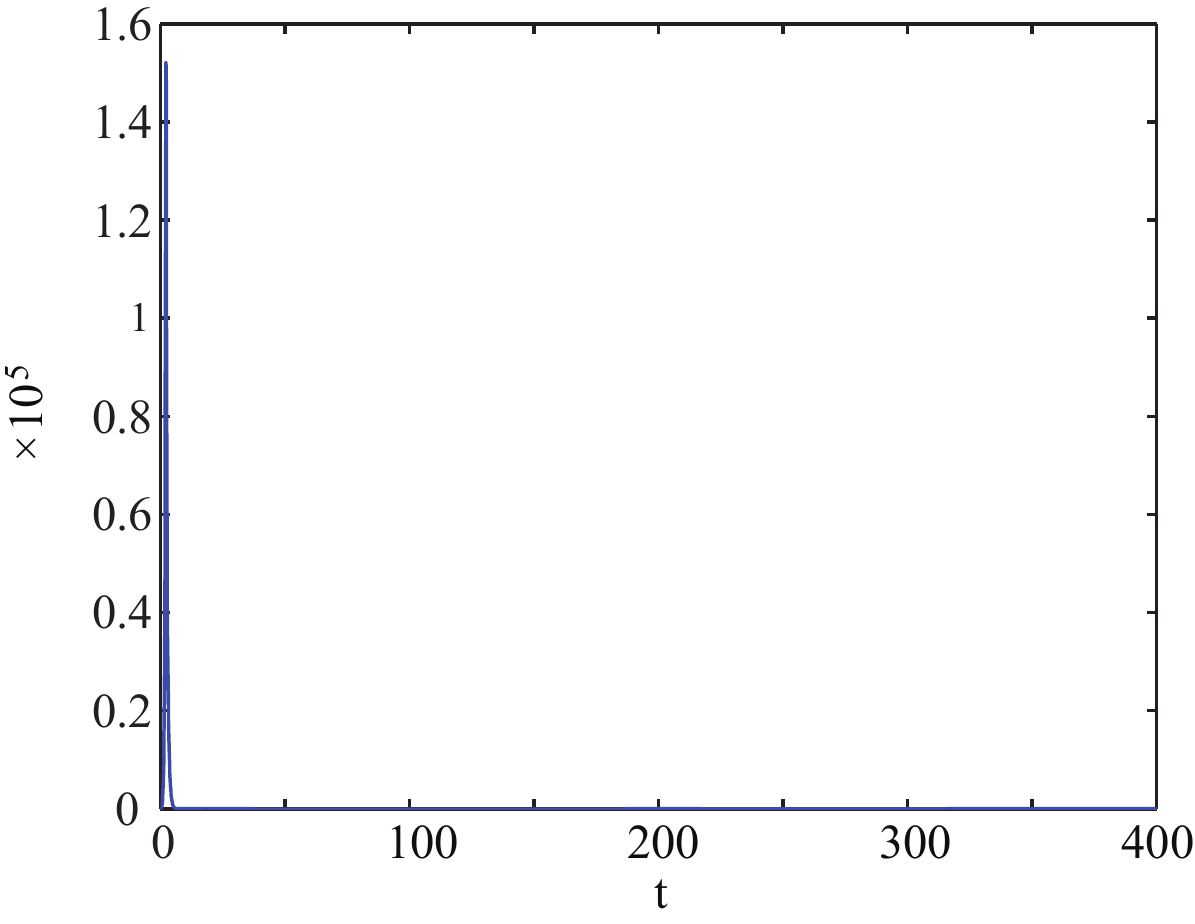}
    \caption{Solutions of the primal (left column) and adjoint (right
      column) fast ODE gating variables measured in the center of $\Omega$.}
    \label{fig:ODEprimaldualfast}
  \end{center}
\end{figure}

\begin{figure}
  \begin{center}
    \includegraphics[height=4.5cm]{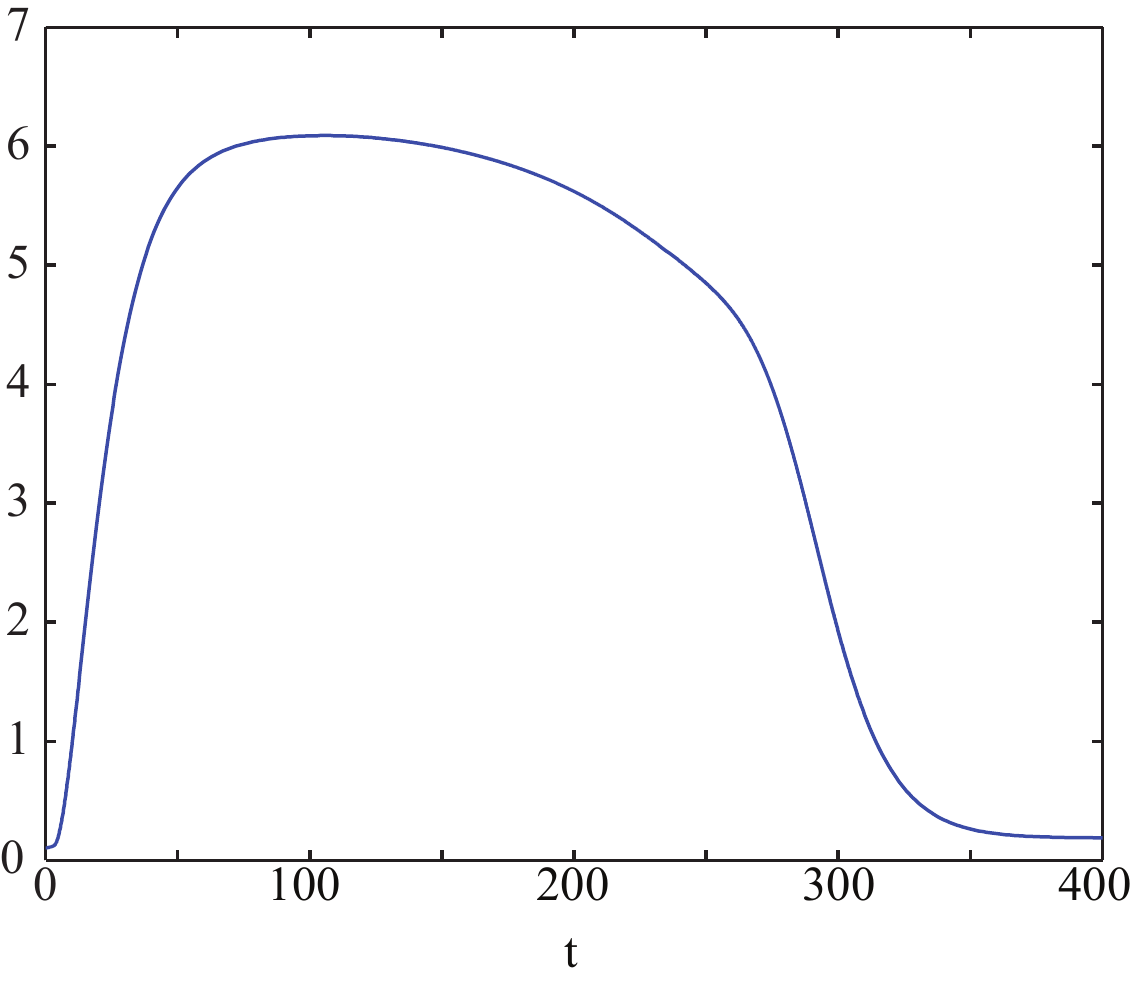} \qquad
    \includegraphics[height=4.5cm]{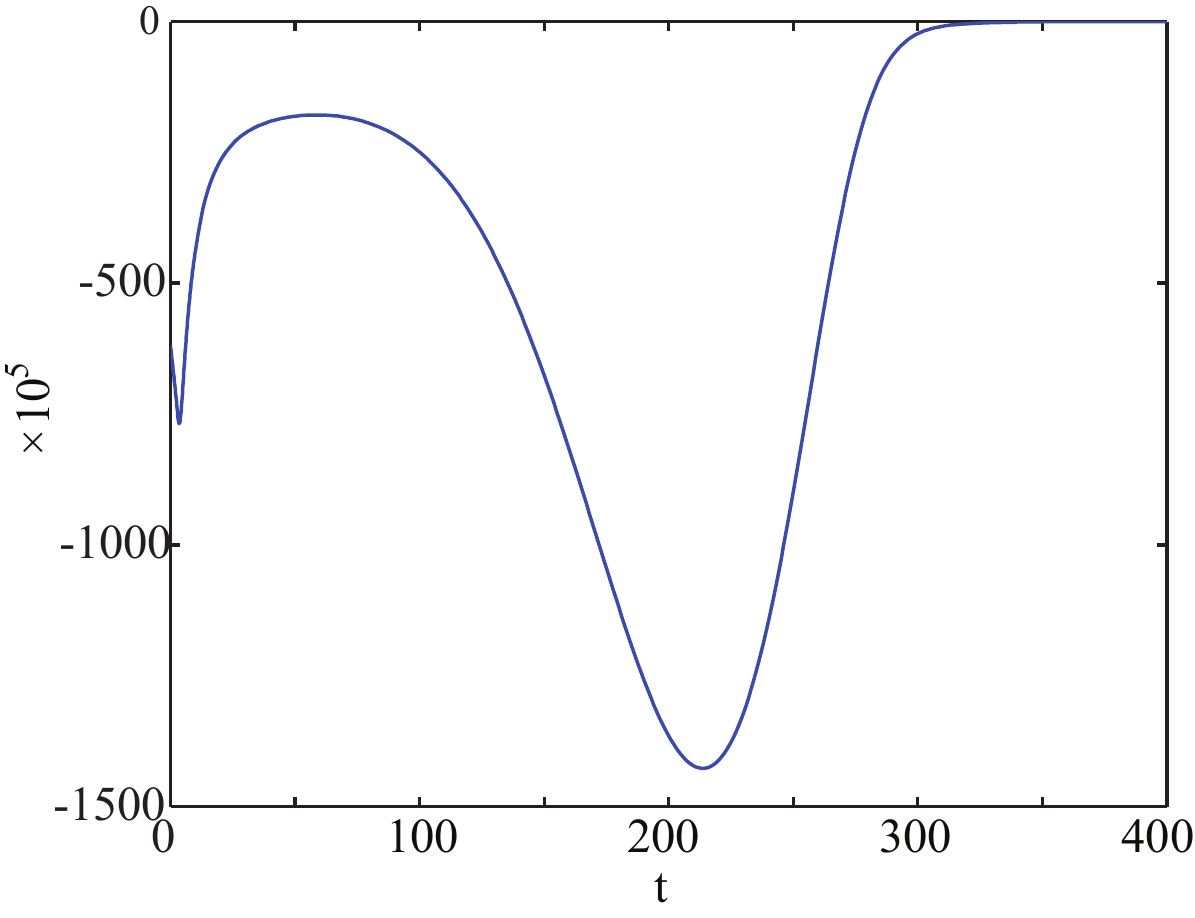}
    \caption{The calcium ion concentration and its adjoint measured in
      the center of $\Omega$. }
    \label{fig:ODEprimaldualcalcium}
  \end{center}
\end{figure}

\begin{figure}
  \begin{center}
    \includegraphics[height=4.5cm]{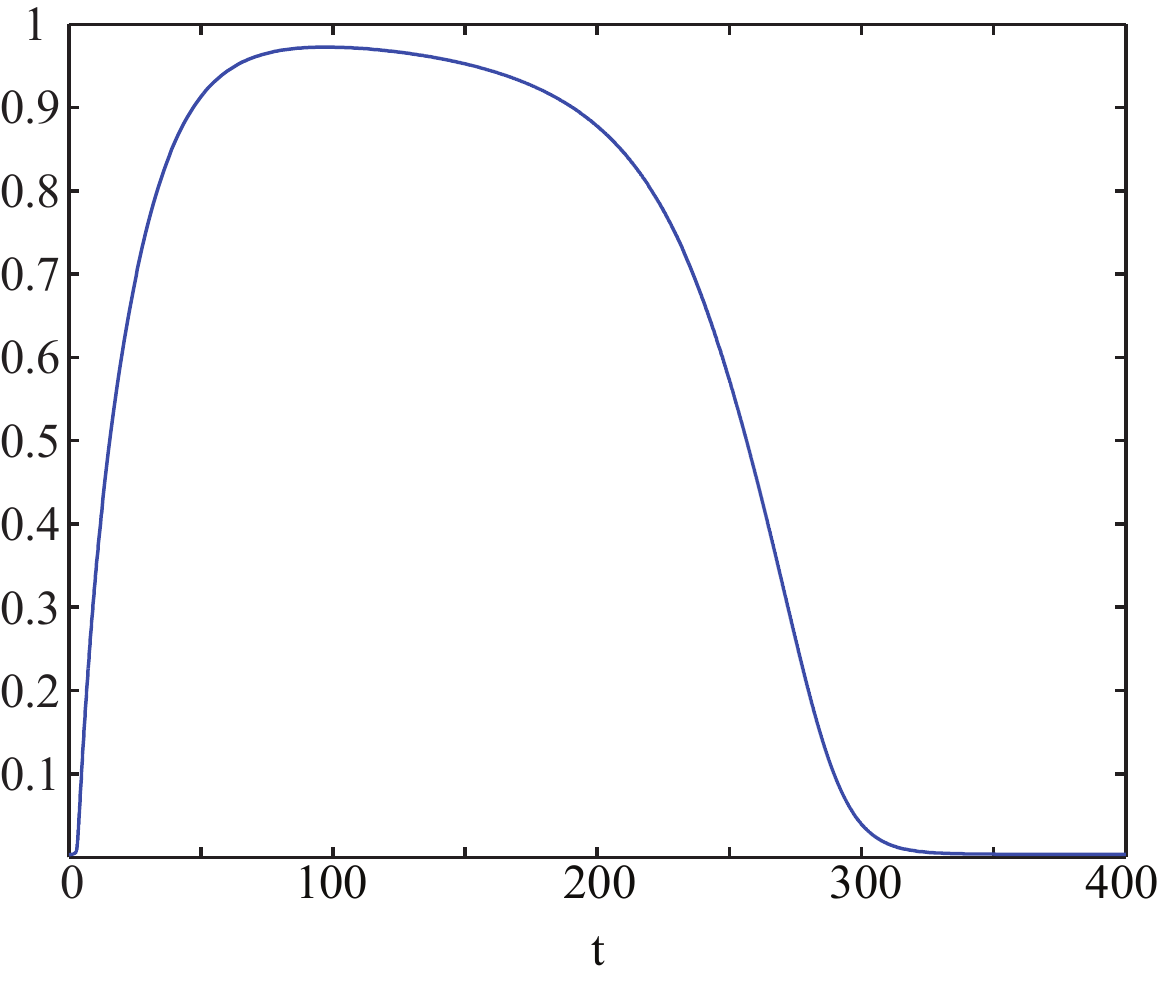} \qquad
    \includegraphics[height=4.5cm]{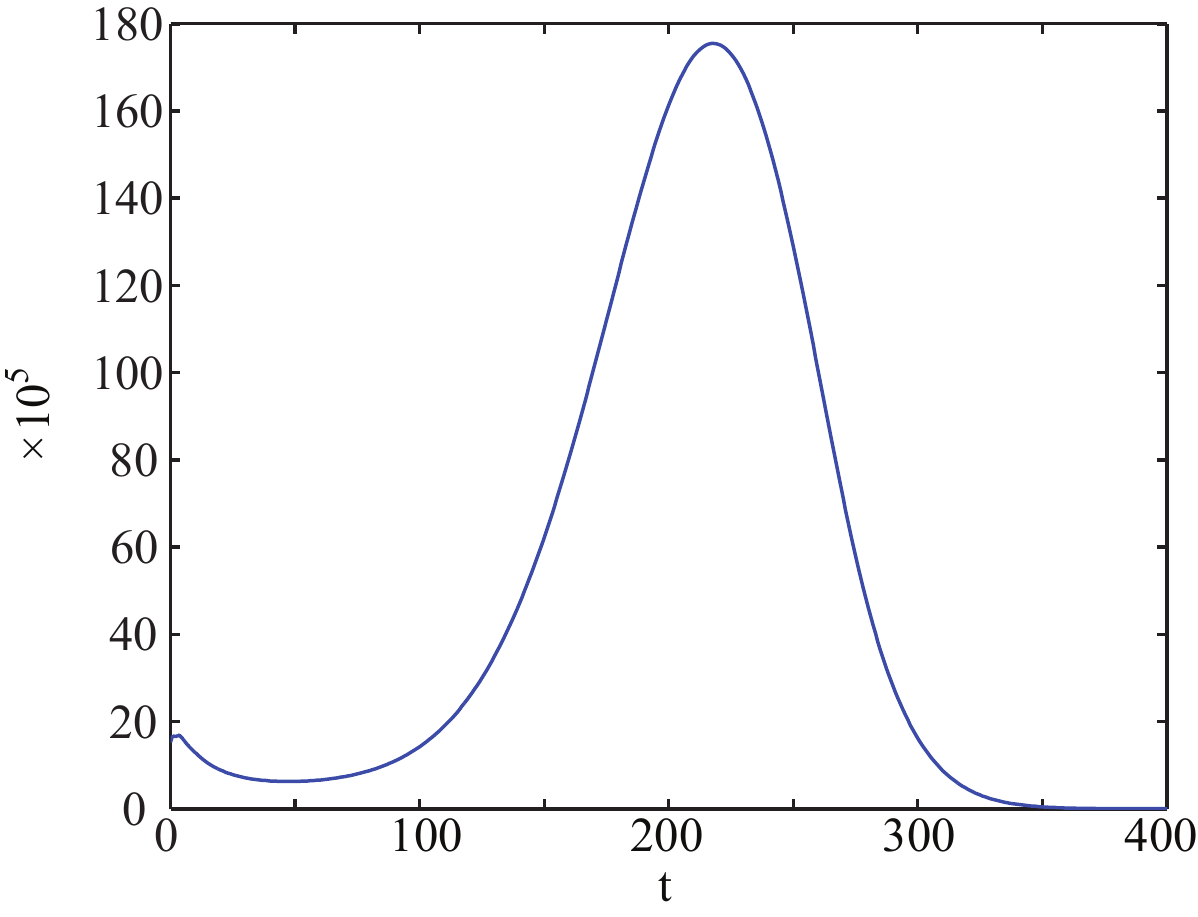}\\
    \includegraphics[height=4.5cm]{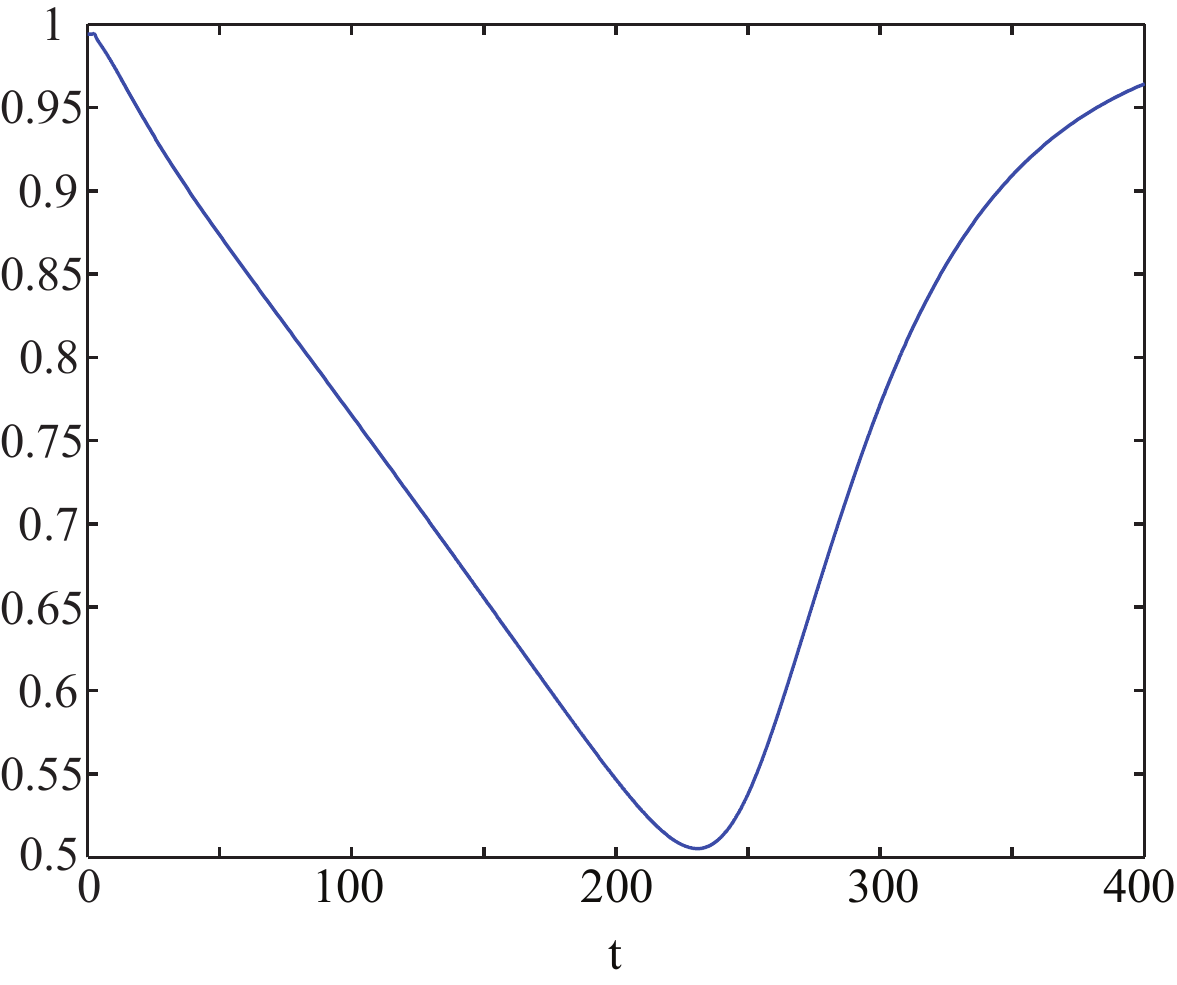} \qquad
    \includegraphics[height=4.5cm]{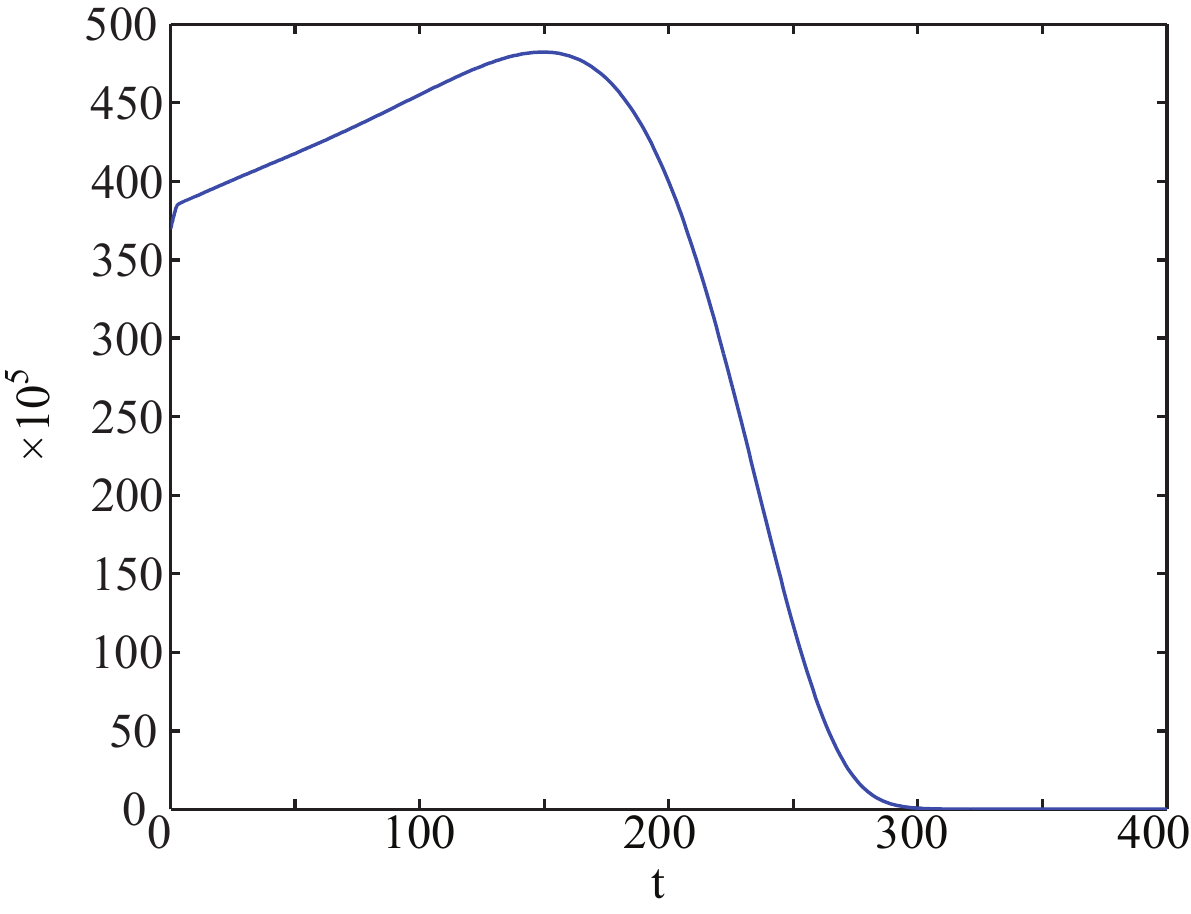}\\
    \includegraphics[height=4.5cm]{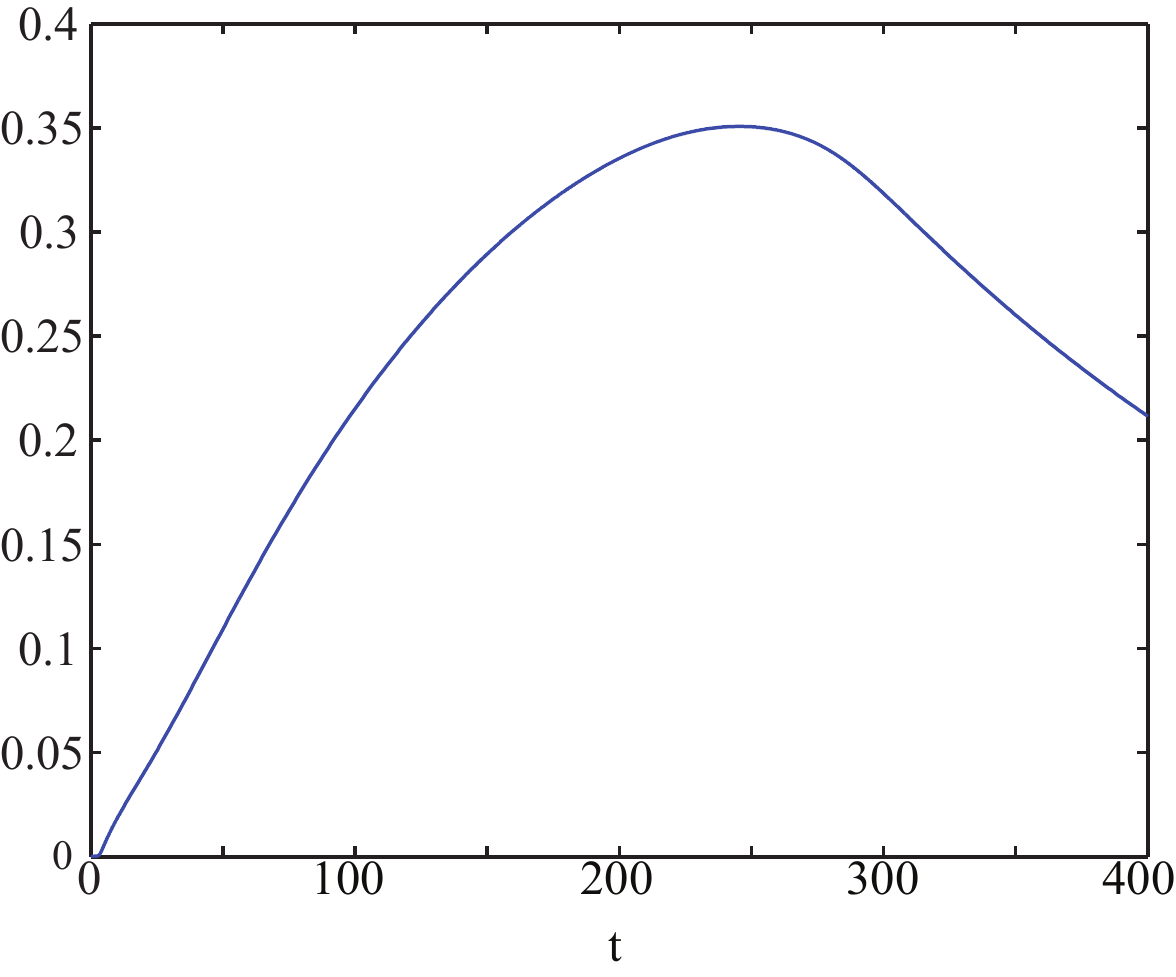} \qquad
    \includegraphics[height=4.5cm]{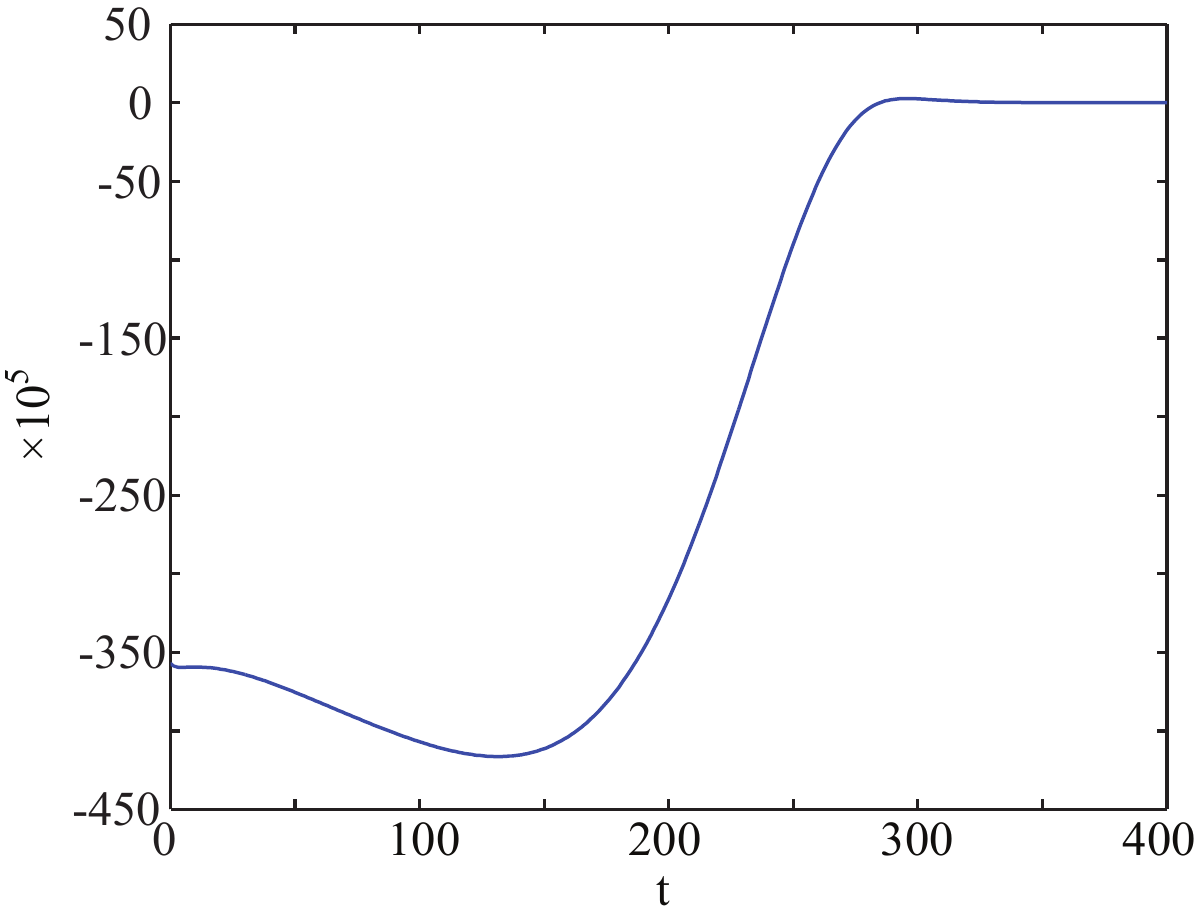}\\
    \caption{Solutions of the primal (left column) and adjoint (right
      column) slow ODE gating variables measured in the center of $\Omega$.}
    \label{fig:ODEprimaldualslow}
  \end{center}
\end{figure}

Recalling that the six of the ODE variables are gating variables and
one models the calcium ion concentration, we can identify two types of
gating variables by looking at Figures \ref{fig:ODEprimaldualfast},
\ref{fig:ODEprimaldualcalcium} and \ref{fig:ODEprimaldualslow}. Three
rapid gating variables in Fig. \ref{fig:ODEprimaldualfast}, which
has a fast change in state in the beginning and then returns to its
original state after about 250-350 ms -- thus slightly later than the
large peak of $\Phi_u$. Moreover, the calcium concentration in Fig.
\ref{fig:ODEprimaldualcalcium} seem to be strongly related to $\Phi_u$
at around 220 ms. Fig. \ref{fig:ODEprimaldualslow} illustrates three
slower gating variables, and similar to the calcium concentration, they
are less influential for small $t$ but rather have their significance
when the system goes back to its resting state.

To examine convergence
properties, we evaluate the error terms after varying the number of
iterations $L_n$ in the iterative multirate method, the spatial and
temporal mesh sizes of the PDE as well as the time steps of the
ODEs. These four parameters were changed uniformly to produce Figures
\ref{fig:errordx} and \ref{fig:errords}.  The default values in these experiments are
$4096$ elements, $\Delta t=\Delta s = 0.1$, $L_n=1$, $\No=10000$ and
$T=20$. As can be seen the errors in the various terms of
\eqref{errorformula2} decrease as expected. The main observation is
that the splitting error Term $V$ dominates the total error, but can
be controlled by increasing $L_n$, cf.\ Fig.
\ref{fig:errords}. In Fig. \ref{fig:Nvcells} we show the effect
of varying the number of Voronoi cells, $\No$, which also show
expected behavior.

\begin{figure}
  \begin{center}
    \includegraphics[scale=0.4]{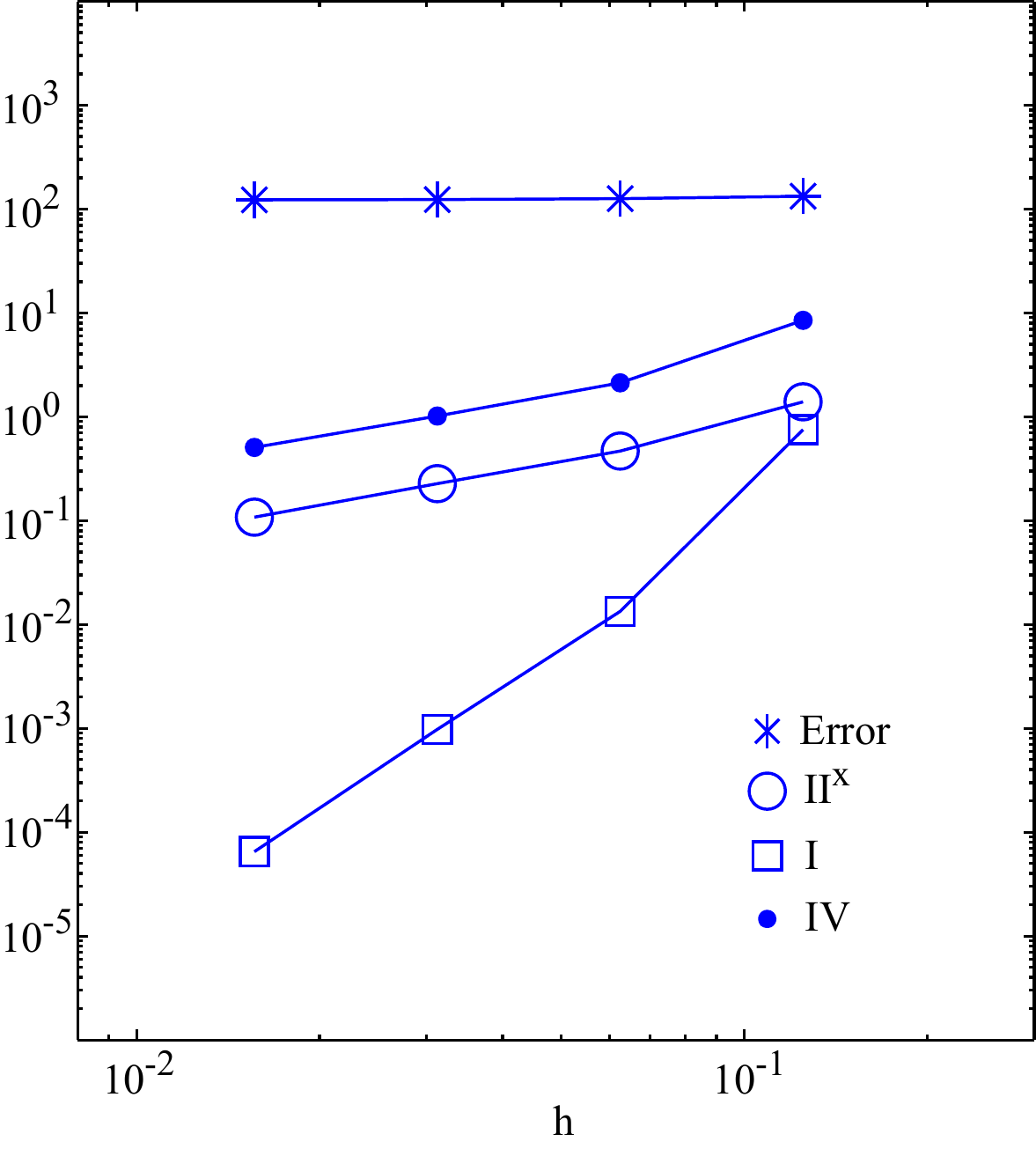}$\qquad$
    \includegraphics[scale=0.4]{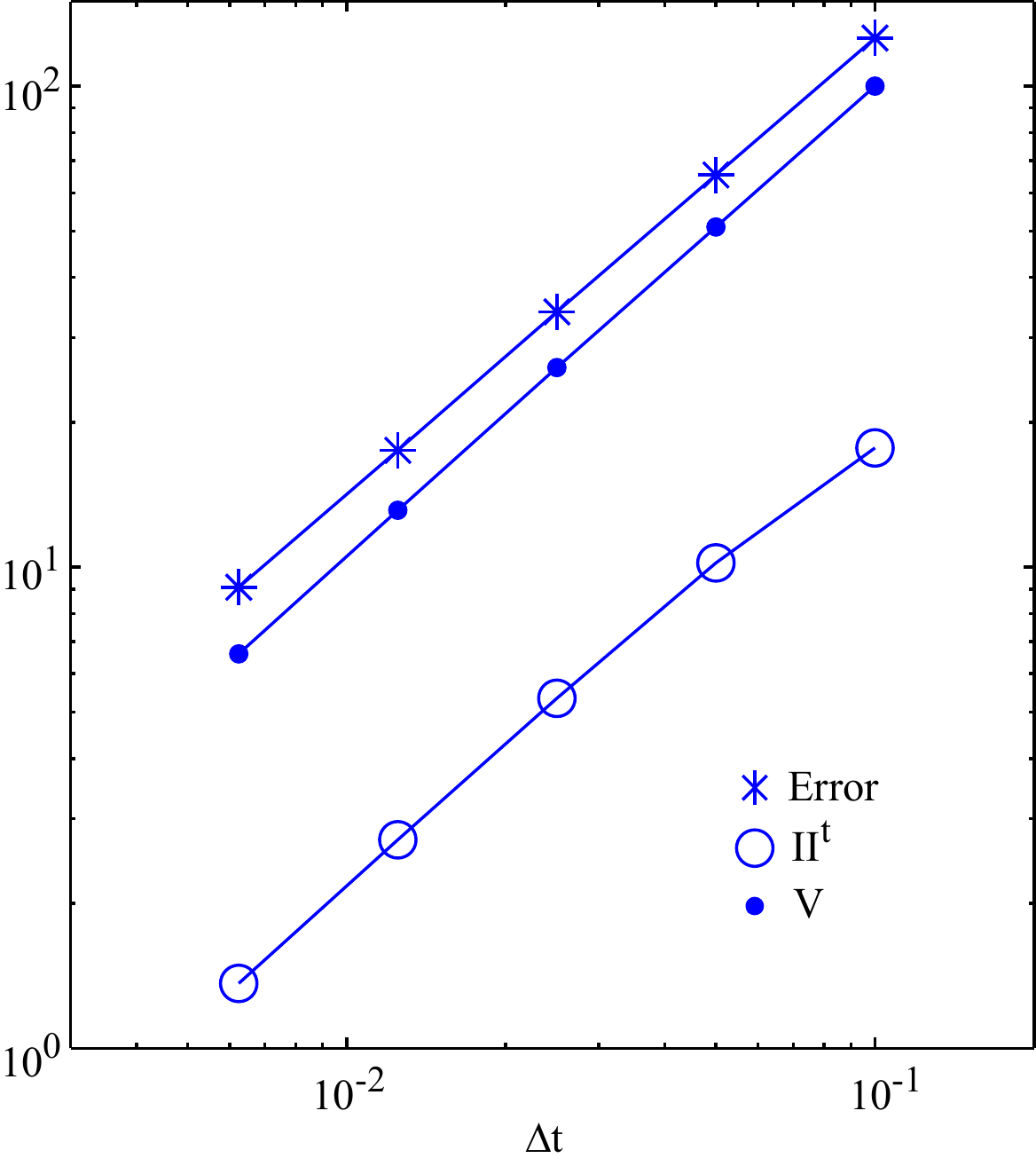}
    \caption{Left: Error and selected contributions as the spatial mesh size varies. Right: Error and selected contributions as the PDE time step varies.}
    \label{fig:errordx}
  \end{center}
\end{figure}

\begin{figure}
  \begin{center}
    \includegraphics[scale=0.4]{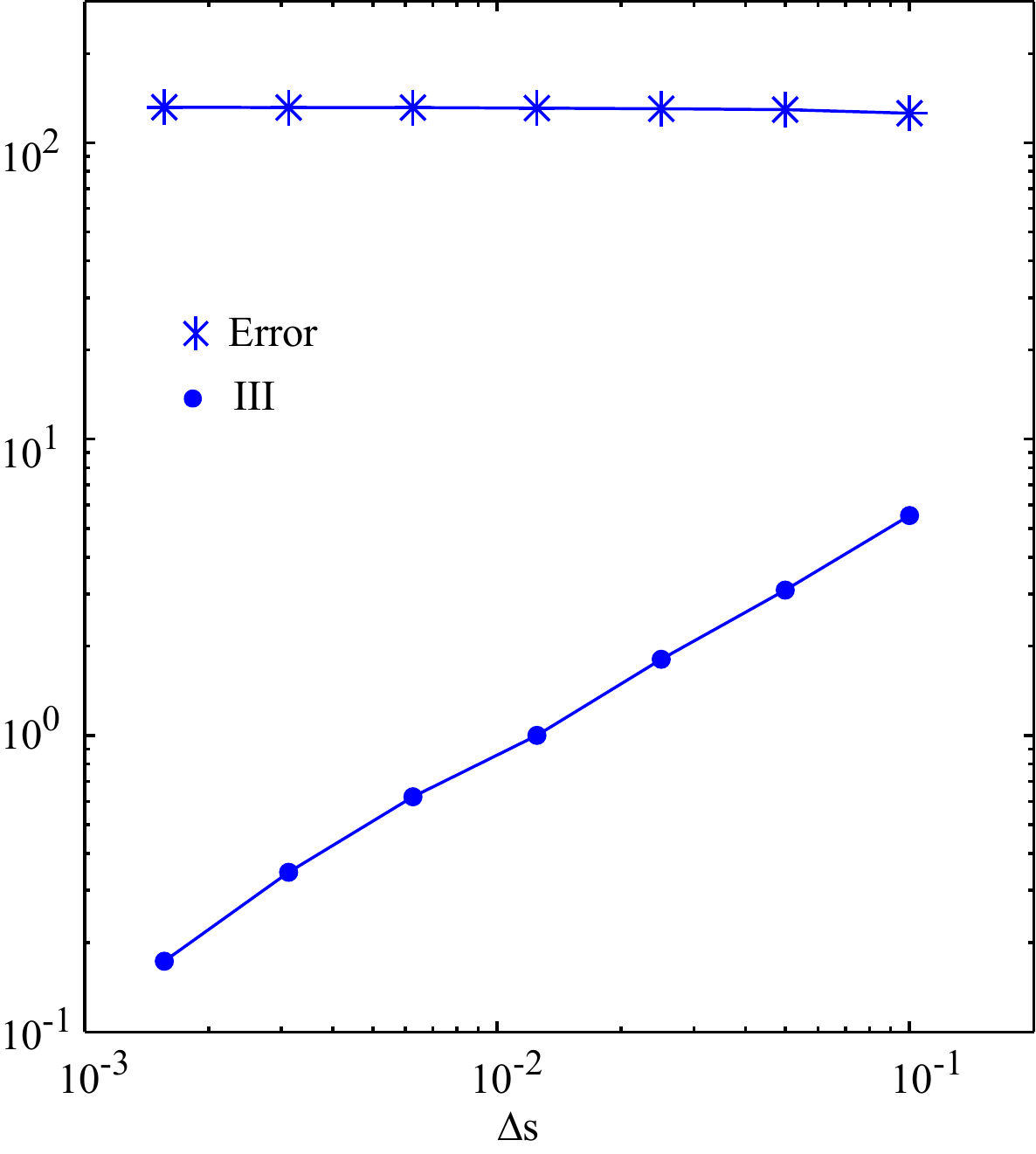}$\qquad$
    \includegraphics[scale=0.4]{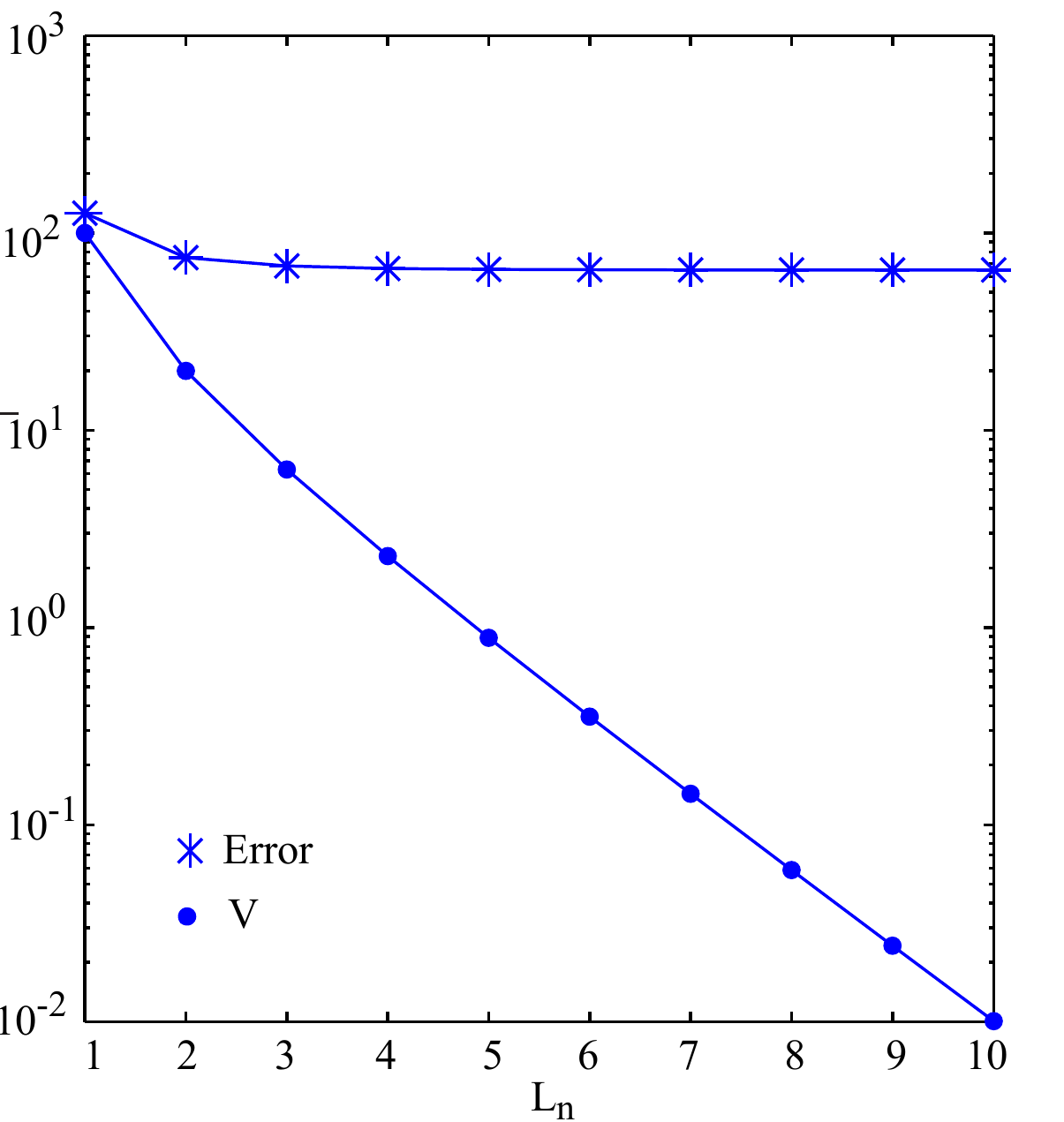}
    \caption{Left: Error and selected contributions as the ODE time step varies. Right: Error and
    selected contribution as the number of iterations $L_n$ in the multirate iterative scheme varies. }
    \label{fig:errords}
  \end{center}
\end{figure}

\begin{figure}
  \begin{center}
    \includegraphics[scale=0.4]{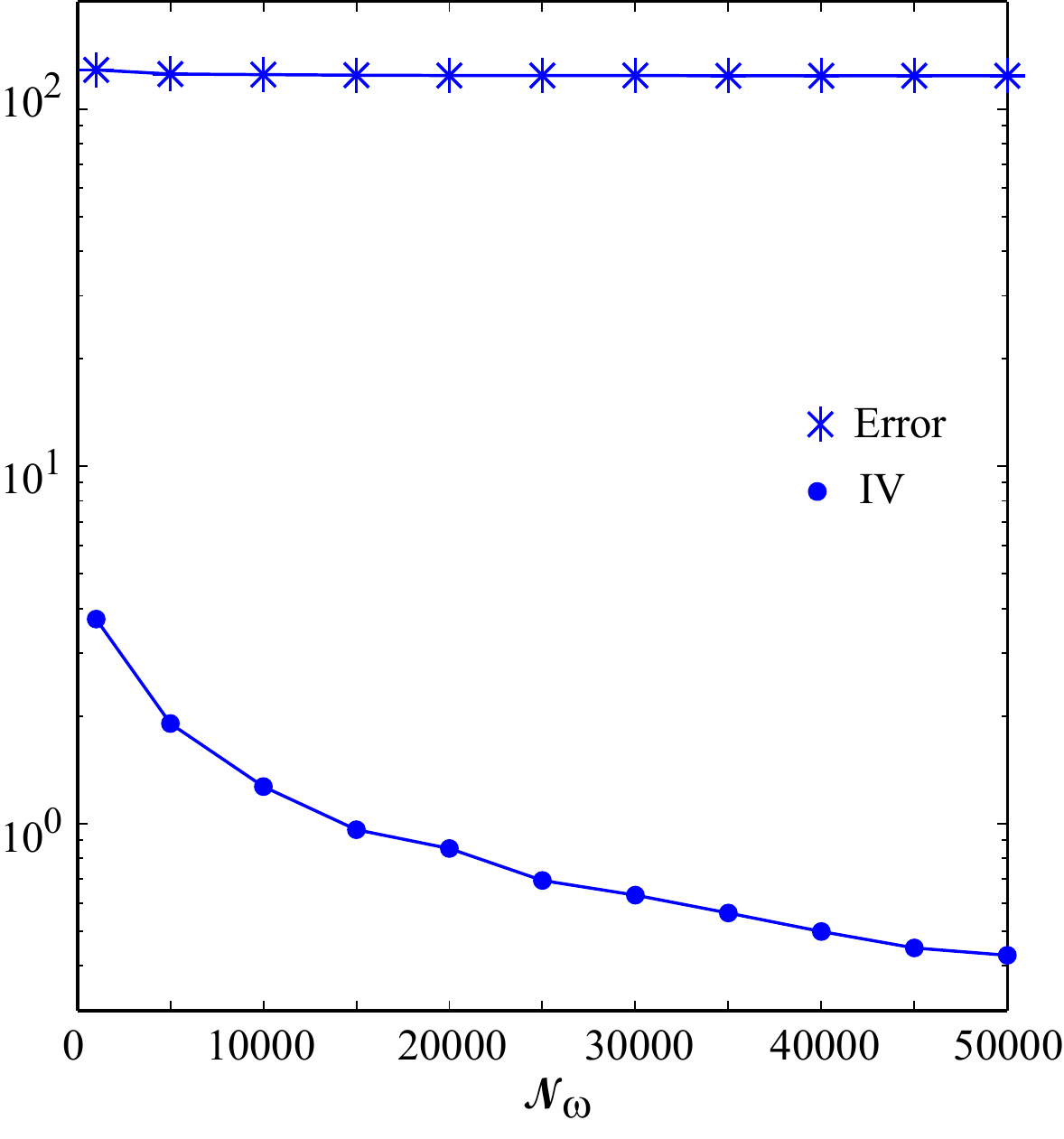}
    \caption{Error and selected contribution as the number of Voronoi cells $\No$ varies.}
    \label{fig:Nvcells}
  \end{center}
\end{figure}

The large dynamics in the system are naturally manifested in the
number of blocks, and the spatial and temporal discretizations on
these blocks. With the given tolerances and parameter values, $8$
blocks are obtained and basic statistics about these blocks can be
found in Table \ref{table:blocks}. As can be seen, the sizes of the
first four blocks are limited by the $\xMAX$ criteria, with predicted
number of elements being $77309$, $79829$, $94074$ and
$87403$. Recalling that $\xMAX=75000$, we thus obtain significantly
more elements. This can be understood by looking at the block creation
procedure in detail. For example, considering block 3, the predicted
number of elements for the first interval $(2.3,2.4]$ was
  $48987$. Since this is less than $\xMAX$, the procedure considers
  also the next interval $(2.4,2.5]$. The predicted number of elements
    for the total interval $(2.3,2.5]$ results in $71870$
      elements. This is also less than $\xMAX$, albeit close. The fact
      that this prediction is close to $\xMAX$ and knowing that the
      dynamics of the problem is rapid could suggest that the block
      should be ended. In the implementation here, the $\xMAX$
      criterion is a strict inequality and the interval $(2.5,2.6]$ is
        considered. This results in a prediction of $94074$ elements
        to guarantee error control. The true number of elements
        obtained are slightly larger, although less than $1$ \%, due
        to the constraint of having at most one hanging node per
        element edge.

The blocks 5, 6, 7 and 8 have the same spatial discretization with
$45718$ elements. Recall that the reason for the spatial meshes not
being coarsened is due to the parameter $\theta$ \eqref{eq:theta}. As
can be seen in Table \ref{table:blocks}, blocks 5, 6 and 7 are limited
in time by the $\tMAX=1000$ criteria. Finally we note that the
variation in time is great: the time step size for the PDE and the
ODEs are in the range of $10^{-3}$ to $1$ ms and $10^{-4}$ to $1$ ms
respectively. Fig. \ref{fig:pdetimesteps} illustrates the various
time steps over time for the PDE.

\begin{table}
  \begin{center}
    \caption{Date on blocks: $b$ is the block number, $T_b$ is the
      end time in ms and $\abs{I_b}$ is the number of time intervals
      in block number $b$. $\Delta t$ is the PDE time step. $M_n$ is the number of ODE time
      subintervals with $\max$ and $\mean$ values. $N_x$ is the number
      of elements that are predicted and actually used.}\label{table:blocks}
    \begin{tabular}{c c c c c c c c c c} \\
      \toprule
      $b$ & $T_b$ (ms) & $\abs{I_b}$ &\multicolumn{3}{c}{$\Delta t\times 10^{-3}$ ($\mu$s)} & \multicolumn{2}{c}{$M_n$} &\multicolumn{2}{c}{$N_x$}\\
      \cmidrule(r){4-6} \cmidrule(r){7-8} \cmidrule(r){9-10}
      & & & $\min$ &
      $\max$ &
      $\mean$ &
      $\max$ &
      $\mean$ &
      predicted &
      used \\
      \midrule
      1 & 1.9 & 352 & 2.27 & 25.0 & 5.41 & 8 & 2.02 & 77309 & 77554\\
      2 & 2.3 & 182 & 1.63 & 3.58 & 2.21 & 6 & 2.07 & 79829 & 79899\\
      3 & 2.6 & 173 & 1.61 & 1.93 & 1.74 & 5 & 2.13 & 94074 & 94158\\
      4 & 2.8 & 120 & 1.51 & 1.89 & 1.68 & 6 & 2.18 & 87403 & 87508\\
      5 & 159.5 & 1000 & 0.78 & 1000 & 157 & 6 & 2.04 & 45718 & 45935 \\
      6 & 266.4 & 1000 & 28.0 & 1000 & 107 & 2 & 2 & 45718 & 45935\\
      7 & 293.7 & 1000 & 22.0 & 56.0 & 27.4 & 2 & 2 & 45718 & 45935\\
      8 & 400 & 222 & 55.0 & 1000 & 481 & 2 & 2 & 45718 & 45935\\
      \bottomrule
    \end{tabular}
  \end{center}
\end{table}

\begin{figure}
  \begin{center}
    \includegraphics[scale=0.5]{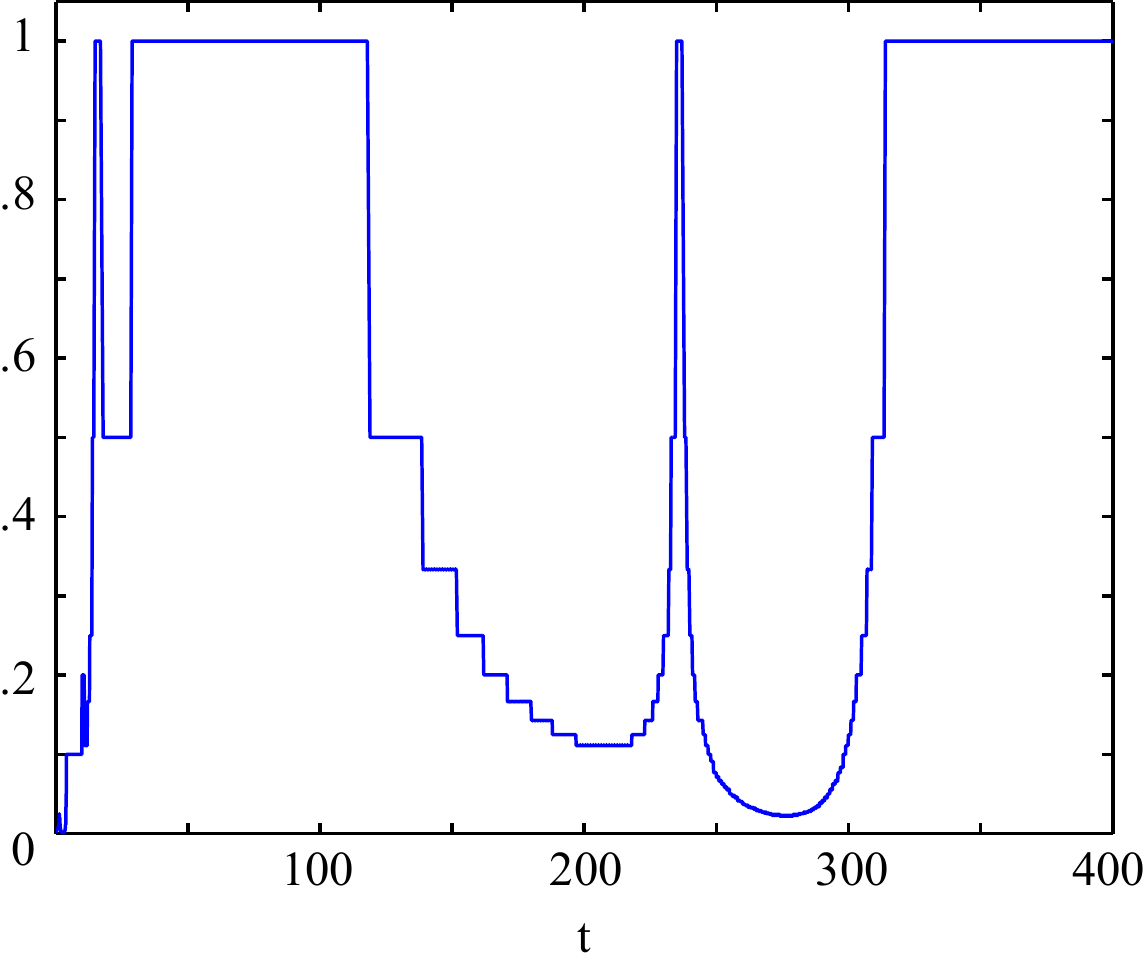}
      \caption{The time steps of the PDE.}
      \label{fig:pdetimesteps}
  \end{center}
\end{figure}


\section{Conclusion}

We consider a problem of a macroscale parabolic PDE which is coupled to a set of
microscale ODEs. We introduce an intermediate scale to couple information between the scales, and use projections to transfer information to the intermediate scale. We use a Monte Carlo method to deal with the very high dimension of the system of ODEs. We also allow the ODEs to be solved on a much finer scale than the PDE. We derive an adjoint-based a posteriori estimate that accounts for all of the key discretization components, and use the estimate to derive indicators of element contributions to the error both in space and time for the PDE
and in time for the ODEs. The estimates take into
account errors in the data passed between the PDE and the ODEs, as
well as the fact that the ODEs are modeled on a much smaller scale
than that of the PDE. The indicators are used to guide an algorithm for adaptive error control.

Finally, we test the adaptive algorithm on a realistic problem.

Future work could consider parallel blockwise adaptivity. Since
the ODEs in this model do not interact inbetween spatial elements,
this set of ODEs constitute an embarrassingly parallel problem and can
simply be parallelized using a graphical processing unit,
GPU. Developing adaptive algorithms designed for modern computer
technologies with several memory hierarchies such as a GPU are indeed
interesting, and the blockwise adaptivity could be one method for
limiting the number of data transfers between hierarchies.

\section*{Acknowledgements}
J.~H.~Chaudhry's work is supported in part by the Department of Energy (DE-SC0005304, DE0000000SC9279).

V.~Carey's work is supported in part by the Department of Energy (DOE-ASCR-1174449-5).

D.~Estep's work is supported in part by the Defense Threat Reduction Agency (HDTRA1-09-1-0036), Department of Energy (DE-FG02-04ER25620, DE-FG02-05ER25699, DE-FC02-07ER54909, DE-SC0001724, DE-SC0005304, INL00120133, DE0000000SC9279), Dynamics Research Corporation PO672TO001, Idaho National Laboratory (00069249, 00115474), Lawrence Livermore National Laboratory (B573139, B584647, B590495),  National Science Foundation (DMS-0107832, DMS-0715135, DGE-0221595003, MSPA-CSE-0434354, ECCS-0700559, DMS-1065046, DMS-1016268, DMS-FRG-1065046, DMS-1228206), and the National Institutes of Health (\#R01GM096192).

V.~Ginting's work is supported in part by the National Science Foundation (DMS-1016283) and the Department of Energy (DE-SC0004982).

M.~Larson's work is supported in part by the Swedish Foundation for Strategic Research Grant (AM13-0029) and the Swedish Research Council Grants (2013-4708,2010-5838).

S.~Tavener's work is supported in part by the Department of Energy (DE-FG02-04ER25620, INL00120133) and National Science Foundation (DMS-1016268).

\bibliographystyle{elsarticle-num}
\bibliography{references}

\end{document}